\documentclass[preprint,11pt]{elsarticle_pre}
\usepackage{hyperref}

\usepackage{float}
\usepackage{amssymb}
\usepackage[dvipsnames]{xcolor}

\usepackage[english]{babel} 

\usepackage{amsmath,amssymb,latexsym,amsthm,enumerate,epsfig,graphicx}%, color}%natbib,
\emergencystretch=\maxdimen
\newtheorem{definition}{Definition}[section]
\newtheorem{theorem}{Theorem}[section]
\newtheorem{proposition}{Proposition}[section]
\newtheorem{lemma}{Lemma}[section]

\newtheorem{remark}{Remark}[section]
\numberwithin{equation}{section}

\newif\ifcomment \commentfalse

\newcommand{\remove}[1]{}

\newenvironment{vardesc]}[1]{%
\settowidth{\parindent}{#1: \ }
\makebox{#1:}}{}

\begin{document} 

\begin{frontmatter}

\title{Well-posedness of a model of nonhomogeneous compressible-incompressible fluids}

%\author{R. Bianchini$^{\diamond}$ \& R. Natalini$^{\star}$}
\author[Bianchini]{Roberta Bianchini}

%\thanks{$^{\diamond}$ Dipartimento di Matematica, Universit\`a degli Studi di Roma "Tor Vergata", via della Ricerca Scientifica 1, I-00133 Rome, Italy.}
%\thanks{$^{\star}$ Istituto per le Applicazioni del Calcolo "M. Picone", Consiglio Nazionale delle Ricerche, via dei Taurini 19, I-00181 Rome, Italy.}

\ead{bianchin@mat.uniroma2.it}
%% \ead[url]{home page}
\fntext[Bianchini]{Dipartimento di Matematica, Universit\`a degli Studi di Roma "Tor Vergata", via della Ricerca Scientifica 1, I-00133 Rome, Italy - Istituto per le Applicazioni del Calcolo "M. Picone", Consiglio Nazionale delle Ricerche, via dei Taurini 19, I-00185 Rome, Italy. }
%\cortext[cor1]{}
%\address{Dipartimento di Matematica, Universit\`a degli Studi di Roma "Tor Vergata", via della Ricerca Scientifica 1, I-00133 Rome, Italy.\fnref{label3}}
 %\fntext[label3]{}

\author[Natalini]{Roberto Natalini}
\ead{roberto.natalini@cnr.it}
%\address{Istituto per le Applicazioni del Calcolo "M. Picone", Consiglio Nazionale delle Ricerche, via dei Taurini 19, I-00181 Rome, Italy.\fnref{label3}}
\fntext[Natalini]{Istituto per le Applicazioni del Calcolo "M. Picone", Consiglio Nazionale delle Ricerche, via dei Taurini 19, I-00185 Rome, Italy.}
\begin{abstract}
We propose a model of a density-dependent compressible-incompressible fluid, which is intended as a simplified version of models based on mixture theory as, for instance, those arising in the study of biofilms, tumor growth and vasculogenesis. Though our model is, in some sense, close to the density-dependent incompressible Euler equations, it presents some differences that require a different approach from an analytical point of view. In this paper, we establish a result of local existence and uniqueness of solutions in Sobolev spaces to our model, using paradifferential techniques. Besides, we show the convergence of both a continuous version of the Chorin-Temam projection method, viewed as a singular perturbation type approximation, and the 'artificial compressibility method'.
\end{abstract}


\end{frontmatter}

\section{Introduction}
\label{intro}

In this paper, we consider a fluid described by the following equations in $\mathbb{R}^{d}$,
\begin{equation}
\label{Goal_System}
\begin{cases}
& \partial_{t}\rho + \nabla \cdot (\rho v) = 0, \\
& \partial_{t}{v}+{v} \cdot \nabla {v} + f(\rho, v) \nabla \rho+ \nabla {P}=0, \\
& \nabla \cdot v = 0,
\end{cases}
\end{equation}
with initial data 
\begin{equation}
\label{initial}
 \rho(0,x)=\rho_{0}(x), ~~~~~~~~~~ v(0,x)=v_{0}(x) ~~\text{such that}~~\nabla \cdot v_{0}(x)=0,
\end{equation}  
where $f(\rho, v)$ is a scalar function of $(\rho, v) \in \mathbb{R}^{d+1}$. System (\ref{Goal_System}) describes the motion of a nonhomogeneous, also called density-dependent, fluid. The nonnegative scalar function $\rho$ is the density of the fluid, $v \in \mathbb{R}^{d}$ its velocity, and $P$ is the incompressible hydrostatic pressure generated by the divergence free constraint. The term $f(\rho, v) \nabla \rho$ is a slight generalization of a compressible pressure. This system is intended as a toy model for a general class of problems arising from mixture theory, which present the coehexistence of the hydrostatic pressure and a compressible pressure term. For instance, consider a simplified version of the model in \cite{cdnr}, composed of just two constituents, a solid(/gel) phase $S$ and a liquid phase $L$:
\begin{equation}
\label{Biofilms_System}
\begin{cases}
& \partial_{t}S + \nabla \cdot (S v_{S}) = \Gamma_{S}, \\
&\partial_{t}L + \nabla \cdot (L v_{L}) = \Gamma_{L}, \\
& \partial_{t}{v}_{S}+{v}_{S} \cdot \nabla {v}_{S} +\frac{\gamma \nabla (1-L)}{S}+ \nabla {P}=\Gamma_{v_{S}}, \\
& \partial_{t}{v}_{L}+{v}_{L} \cdot \nabla {v}_{L} + \nabla {P}=\Gamma_{v_{L}}, \\
& S+L=1, \\
& \Gamma_{S}+\Gamma_{L}=0,
\end{cases}
\end{equation}
where $v_{S}, v_{L}$ are the velocities of the solid and the liquid phase respectively, $\Gamma_{S}, \Gamma_{L}, \Gamma_{v_{S}}, \Gamma_{v_{L}}$ are source terms and $\gamma$ is an experimental constant. Using the last two conservation constraints, system (\ref{Biofilms_System}) can be reduced to 
\begin{equation}
\label{Biofilms_System1}
\begin{cases}
& \partial_{t}S + \nabla \cdot (S v_{S}) = \Gamma_{S}, \\
& \partial_{t}{v}_{S}+{v}_{S} \cdot \nabla {v}_{S} +\frac{\gamma \nabla S}{S}+ \nabla {P}=\Gamma_{v_{S}}, \\
& \partial_{t}{v}_{L}+{v}_{L} \cdot \nabla {v}_{L} + \nabla {P}=\Gamma_{v_{L}}, \\
& \nabla \cdot (Sv_{S} + (1-S)v_{L}) = 0, \\
& L=1-S,
\end{cases}
\end{equation}
where the equation for the solid phase velocity $v_{S}$ presents a pressure term composed by two parts, one incompressible part $\nabla P$ and the compressible one $\frac{\gamma \nabla (S)}{S} = \gamma \nabla log(S)$. More generally, these kinds of problems, which are characterized by the interaction between compressible and incompressible pressure terms, arise from \emph{mixture theory}, as, for instance, models of biofilms \cite{cdnr}, tumor growth \cite{astanin} and organic tissues and vasculogenesis \cite{Farina}. Anyway, models deriving from mixture theory are similar to system (\ref{Biofilms_System1}), by replacing the compressible pressure $\gamma log(S)$ with a more general function $P(S)=\phi(S)$ only depending on the solid phase. As a matter of facts, model (\ref{Biofilms_System1}) presents several analytical difficulties, which we are trying to understand by studying a simplified version. In order to do this, the first idea is to consider a model where the solid phase $S$ and the liquid $L$ have the same transport velocity $v=v_{S}=v_{L}$, and whose equation contains a compressible pressure term. These assumptions give the following model:

\begin{equation}
\label{density_dependent_ours}
\begin{cases}
& \partial_{t}\rho + \nabla \cdot (\rho v) = 0, \\
& \partial_{t}{v}+{v} \cdot \nabla {v} + \nabla \phi(\rho) + {\nabla {P}}=0, \\
& \nabla \cdot v = 0,
\end{cases}
\end{equation}

where $\rho$ is the density of the fluid and $\phi(\rho)$ a general compressible pressure. Actually, by defining a new pressure term $Q:=P+\phi(\rho)$, model (\ref{density_dependent_ours}) can be reduced to

\begin{equation}
\label{reduced}
\begin{cases}
& \partial_{t}\rho + \nabla \cdot (\rho v) = 0, \\
& \partial_{t}{v}+{v} \cdot \nabla {v} + \nabla Q=0, \\
& \nabla \cdot v = 0, \\
& P=Q-\phi(\rho), \\
\end{cases}
\end{equation}

which is just the homogeneous incompressible Euler equations plus a transport equation for the density variable and, although the techniques developed in this paper continue to work, it can be solved in a trivial way. Therefore, we are going to study a mathematical generalization of model (\ref{density_dependent_ours}), represented by our system in (\ref{Goal_System}), which has got most of the analytical difficulties of model (\ref{Biofilms_System1}). Let us notice that system (\ref{Goal_System}) is somewhat similar to the density-dependent incompressible Euler equations
\begin{equation}
\label{density_dependent}
\begin{cases}
& \partial_{t}\rho + \nabla \cdot (\rho v) = 0, \\
& \partial_{t}{v}+{v} \cdot \nabla {v} + \frac{\nabla {P}}{\rho}=0, \\
& \nabla \cdot v = 0,
\end{cases}
\end{equation}
which have been studied by many authors, see for instance \emph{J. E. Marsden} \cite{Marsden}, \emph{H. Beir\~{a}o da Veiga} \cite{Da Veiga}, \emph{A. Valli} \cite{Valli} \& \emph{R. Danchin} \cite{Danchin}. Let us remark that, in \cite{Valli}, \emph{Valli} and \emph{Zajaczkowski} have studied  model (\ref{density_dependent}) by using an approximating system where the divergence of the velocity field gradually vanishes in a similar way to the {Chorin-Temam projection method} in \cite{Temam}. Although model (\ref{Goal_System}) looks quite similar to the density-dependent model (\ref{density_dependent}), it happens that, apart from the vorticity method, which we have still not fully explored, most of the ideas used to approach it do not apply to our system. The main problem is that we cannot expect to gain one more space derivative for the pressure term $P$ with respect to the regularity of the other unknowns $\rho, v$, as it is in the case of classical incompressible fluids. This can be viewed by applying the divergence operator to the velocity equation in (\ref{Goal_System}), which yields
\begin{equation}
\label{Pressure_elliptic}
\Delta P + \nabla \cdot (f(\rho, v) \nabla \rho))= - \nabla v \nabla v.
\end{equation}
Now, if  $\rho \in C([0,T], H^{m}(\mathbb{R}^{d}))$, then $P \in C([0,T], H^{m}(\mathbb{R}^{d}))$, namely $\nabla P \in C([0,T], H^{m-1}(\mathbb{R}^{d}))$. On the other hand, if we look at the method of resolution of the density-dependent Euler equations in \cite{Valli}, this leads to consider the following elliptic equation in the pressure term
\begin{equation}
\nabla \cdot \Bigg(\frac{\nabla P}{\rho}\Bigg) = - \nabla v \nabla v.
\end{equation}
In this case, the regularity of $\rho$ and $v$ guarantees, if $m > [d/2]+1$, one more space derivative of regularity for the pressure term, namely: if $\rho, v$ belong to $ C([0,T], H^{m}(\mathbb{R}^{d}))$, then $P$ belongs to $C([0,T], H^{m+1}(\mathbb{R}^{d}))$. Conversely, considering system (\ref{Goal_System}) and the related elliptic equation (\ref{Pressure_elliptic}), we realize that it is not easy to obtain energy estimates on equation (\ref{Pressure_elliptic}), since we have not enough regularity in terms of the incompressible pressure $P$. This is exactly what happens for system (\ref{Biofilms_System1}). Nevertheless, here we estabilish the well-posedeness of system (\ref{Goal_System}), using an approximation based on paradifferential calculus. Next, we show the convergence of a new singular perturbation approximation that can be considered as a continuous - in time - version of the {projection method} in \cite{Temam}, which turns out to work also on the homogeneous incompressible Euler equations. Finally, we briefly show that also the more classical artificial compressibility method in \cite{Temam} works on system (\ref{Goal_System}). We point out that these three methods - the application of paradifferential calculus, our continuous version of the projection method and the adapted artificial compressibility, which we have applied in order to prove well-posedness of system (\ref{Goal_System}) - do not seem to work on the density-dependent Euler equations, see Remarks \ref{remark1}, \ref{remark3}, and \ref{remark2}. 

\subsection{Organization of the paper}
The paper is organised as follows. First, we explain some structural characteristics of system (\ref{Goal_System}), which will be useful in the following. Section 2 is devoted to the definition of the first type of approximation based on paradifferential operators and the related proof of existence and uniqueness. In Section 3, we prove the convergence of our continuous version of the projection method. Finally, in Section 4, we shortly discuss the artificial compressibility approximation and its convergence.

\begin{remark}
\label{First_Remark}
Notice that, when the function $f(\rho, v) = f(\rho)$ only depends on the density variable $\rho$, the velocity equation in (\ref{Goal_System}), with the divergence free condition $\nabla \cdot v = 0$, leads to system (\ref{density_dependent_ours}) and then to (\ref{reduced}), where the homogeneous incompressible Euler equations can be solved in order to get the unknown density $\rho$ by its transport equation. Therefore, in that case we have smooth solutions to system (\ref{Goal_System}), in accordance with the regularity of the solutions to the homogeneous incompressible Euler equations, which can be seen in \cite{Lions}.
\end{remark}

\section{General setting}
Let $\textbf{u}=(\rho, v)$ and $F_{P}=(0, \nabla P)^{T}$. System (\ref{Goal_System}) can be written in the compact form 

\begin{equation}
\label{Goal_Compact}
\begin{cases}
& \partial_{t} \textbf{u}+\sum_{j=1}^{d}A_{j}(\textbf{u})\partial_{x_{j}}\textbf{u}+F_{P}=0, \\
& \nabla \cdot v = 0, 
\end{cases}
\end{equation}

with initial data (\ref{initial})

\begin{equation}
\label{Compact_initial_data}
\textbf{u}(0, x)=\textbf{u}_{0}(x)=(\rho_{0}(x), v_{0}(x))^{T},
\end{equation}

where, in the $2$-dimensional case

\begin{equation}
\label{A11}
A_{1}(\textbf{u}) =  \left( \begin{array}{ccc}
v_{1} & \rho & 0 \\
f(\textbf{u}) & v_{1} & 0 \\
0 & 0 & v_{1} 
\end{array} \right), \ 
A_{2}(\textbf{u}) =  \left( \begin{array}{ccc}
v_{2} & 0 & \rho\\
0 & v_{2} & 0 \\
f(\textbf{u})  & 0 & v_{2} 
\end{array} \right),
\end{equation}

and, in the general $d$-case, for $j=1, \cdots, d,$

\begin{equation}
\label{Ajj}
A_{j}(\textbf{u}) =  \left( \begin{array}{ccccc}
v_{j} & \delta_{1j}\rho & \delta_{2j}\rho & \cdots & \delta_{dj}\rho\\
\delta_{j1} f(\textbf{u}) & v_{j} & 0 & \cdots & 0 \\
\delta_{j2} f(\textbf{u}) & 0 & v_{j} & \cdots & 0 \\
\cdots &  \cdots & \cdots &  v_{j} & \cdots \\
\delta_{jd} f(\textbf{u}) & 0 & 0 & \cdots & v_{j} \\
\end{array} \right).
\end{equation}

The symbol $A(\xi, \textbf{u}):=\sum_{j=1}^{d} A_{j}(\textbf{u})\xi_{j}$ associated to the paradifferential operator related to system (\ref{Goal_Compact}) is

\begin{equation}
\label{A_symbol}
A(\xi, \textbf{u}) =  \left( \begin{array}{ccccc}
\sum_{j=1}^{d}v_{j}\xi_{j} & \rho \xi_{1} & \rho\xi_{2} & \cdots & \rho_{\xi_{d}}\\
f(\textbf{u})\xi_{1} & \sum_{j=1}^{d}v_{j}\xi_{j} & 0 & \cdots & 0 \\
f(\textbf{u})\xi_{2} & 0 & \sum_{j=1}^{d}v_{j}\xi_{j} & \cdots & 0 \\
\cdots &  \cdots & \cdots &  \sum_{j=1}^{d}v_{j}\xi_{j} & \cdots \\
f(\textbf{u})\xi_{d} & 0 & 0 & \cdots & \sum_{j=1}^{d}v_{j}\xi_{j}\\
\end{array} \right).
\end{equation}

We get the eigenvalues

\begin{itemize}
\item $\lambda_{1} = \cdots =\lambda_{d-1}=\sum_{j=1}^{d}v_{j}\xi_{j},$
\item $\lambda_{d} = \sum_{j=1}^{d}v_{j}\xi_{j} - \sqrt{f(\textbf{u}) \rho} |\xi|,$
\item $\lambda_{d+1} = \sum_{j=1}^{d}v_{j}\xi_{j} + \sqrt{f(\textbf{u}) \rho} |\xi|,$
\end{itemize}

and the related eigenvectors

\begin{itemize}
\item $\textbf{e}_{1} = (0, -\xi_{2}, \xi_{1}, 0, \cdots, 0)^{T}$,
\item $\textbf{e}_{2} = (0, -\xi_{3}, 0, \xi_{1}, 0, \cdots, 0)^{T}$,
\item $\textbf{e}_{j}=(0, -\xi_{j+1}, 0, \cdots, 0, \xi_{1}, 0, \cdots, 0)^{T}$,
\item $\textbf{e}_{d-1}=(0, -\xi_{d}, 0, \cdots, 0, \xi_{1})^{T}$,
\item $\textbf{e}_{d} = (-\sqrt{\rho}|\xi|,  \sqrt{f(\textbf{u})}\xi_{1}, \sqrt{f(\textbf{u})}\xi_{2}, \cdots,  \sqrt{f(\textbf{u})}\xi_{d})^{T}$,
\item $\textbf{e}_{d+1} = (\sqrt{\rho}|\xi|, \sqrt{f(\textbf{u})}\xi_{1},  \sqrt{f(\textbf{u})}\xi_{2}, \cdots, \sqrt{f(\textbf{u})}\xi_{d})^{T}$.
\end{itemize}

\begin{remark} \textbf{Assumptions}:
\label{assumptions}
notice from above that, in order to have real and semisimple eigenvalues, which is an essential property to guarantee hyperbolicity (see \cite{Metivier}), we have to assume $f(\textbf{u})$ strictly positive.
\end{remark}

Now, let us go back to the general setting of system (\ref{Goal_Compact}). First, let us neglect for a while the incompressible pressure term $F_{P}$ in (\ref{Goal_Compact}). Therefore, it is easy to check that we are considering a Friedrichs symmetrizable hyperbolic system, whose (positive definite, since Remark \ref{assumptions} and Remark \ref{Translation}) symmetrizer is the diagonal $(d+1) \times (d+1)$ matrix

\begin{equation}
\label{Symmetrizer1}
A_{0}(\textbf{u})=diag\Bigg(\frac{f(\textbf{u})}{\rho}, 1, 1, \cdots, 1\Bigg).
\end{equation}

To clarify the calculations below, we write the explicit expression of

\begin{equation}
\label{Symm_Aj}
A_{0}A_{j} = \left( \begin{array}{ccccc}
\frac{f(\textbf{u}) v_{j}}{\rho} & \delta_{1j}f(\textbf{u}) & \delta_{2j}f(\textbf{u}) & \cdots & \delta_{dj}f(\textbf{u}) \\
\delta_{j1} {f(\textbf{u})} & v_{j} & 0 & \cdots & 0 \\
\delta_{j2} {f(\textbf{u})} & 0 & v_{j} & \cdots & 0 \\
\cdots &  \cdots & \cdots &  v_{j} & \cdots \\
\delta_{jd} {f(\textbf{u})} & 0 & 0 & \cdots & v_{j} \\
\end{array} \right),
\end{equation}
for $j=1, \cdots d$.

\begin{remark}
\label{Translation}
Since (\ref{Symmetrizer1}) and (\ref{Symm_Aj}), we require that the scalar variable $\rho$ is not vanishing  for every $(x,t)$. This can be obtained by recalling that the density equation in (\ref{Goal_System}) can be written as
$$\displaystyle \partial_{t}\rho+\nabla \rho \cdot v = 0.$$
Therefore, if the initial datum $\rho_{0}$ in (\ref{initial}) is not vanishing for all $x \in \mathbb{R}^{d}$, then $\rho(t,x)$ cannot vanish under some standard assumptions of regularity. In particular, if the initial datum for the density  $\rho_{0}$ is in $W^{1,\infty}$ and  $v \in L^{1}([0,T], Lip(\mathbb{R}^{d}))$, the positivity of $\rho$ follows by the results in \cite{Lad}. In the following, we are going to prove that, fixing a constant value $\bar{\rho}$, if we take $\rho_{0}$ such that $\rho_{0}-\bar{\rho} \in H^{m}(\mathbb{R}^{d})$, with $m>[d/2]+1$, then $(\rho-\bar{\rho}, v) \in C([0,T], H^{m}(\mathbb{R}^{d})) \cap C^{1}([0,T], H^{m-1}(\mathbb{R}^{d}))$, and so we are in the assumptions of Proposition 1 in \cite{Lad}. Finally, notice that the non-vanishing density variable is a condition also required by the particular case of system (\ref{density_dependent_ours}) where, as in (\ref{Biofilms_System1}), the compressible pressure $\nabla \Phi(\rho)= \frac{\gamma \nabla \rho}{\rho}$ and, overall, by the general system (\ref{Biofilms_System1}), where $S$ is the considered density. This last observation shows that, although explicitly shown only by the expression of the symmetrizer (\ref{Symmetrizer1}), this is an intrinsic hypothesis of models (\ref{Goal_System}), (\ref{density_dependent_ours}) and (\ref{Biofilms_System1}).
\end{remark}

At this point, applying the symmetrizer $A_{0}(\textbf{u})$ to system (\ref{Goal_Compact}), we obtain the symmetric formulation

\begin{equation}
\label{Symmetric_Compact}
A_{0}(\textbf{u})\partial_{t}\textbf{u}+\sum_{j=1}^{d}A_{0}A_{j}(\textbf{u})\partial_{x_{j}}\textbf{u}+A_{0}(\textbf{u})F_{P}=0.
\end{equation}

We point out that

\begin{equation}
\label{Reasonable}
A_{0}(\textbf{u})F_{P}= diag\Bigg(\frac{f(\textbf{u})}{\rho}, 1, 1, \cdots, 1\Bigg) \cdot (0, \nabla P)^{T} = (0,\nabla P)^{T}=F_{P}.
\end{equation}

This means that the $A_{0}$-scalar product preserves the gradient function $\nabla P$ and this essential fact makes possible to merge energy estimates induced by the symmetrizer, as in \cite{Majda}, \cite{Metivier}, \cite{Benzoni}, with the Leray projector in \cite{Bertozzi}, \cite{Klainerman}, \cite{Temam}.
We rewrite equation (\ref{Symmetric_Compact}) using (\ref{Reasonable}). Then, we have

\begin{equation}
\label{Symmetric}
A_{0}(\textbf{u})\partial_{t}\textbf{u}+\sum_{j=1}^{d}A_{0}A_{j}(\textbf{u})\partial_{x_{j}}\textbf{u}+F_{P}=0.
\end{equation}

Now, following \cite{Bertozzi}, \cite{Klainerman}, \cite{Temam}, we project equation (\ref{Symmetric}) onto the space of the divergence free velocity $v$. Namely, setting

\begin{equation}
\label{Projector}
\textbf{P}= \left( \begin{array}{cccc}
Id & 0 \\
0 &  \mathbb{P} \\
\end{array} \right),
\end{equation}

where $\mathbb{P}$ is the standard Leray projector and, applying the projector operator (\ref{Projector}) to (\ref{Symmetric}), this last equation becomes

\begin{equation}
\label{System_Pressure_Free}
\textbf{P}(A_{0}\partial_{t}\textbf{u}+\sum_{j=1}^{d}A_{0}A_{j}(\textbf{u})\partial_{x_{j}}\textbf{u})=0.
\end{equation}

By definition (\ref{Projector}), $$\displaystyle \textbf{P}(A_{0}\partial_{t}\textbf{u}) = \Bigg(\frac{f(\textbf{u})}{\rho} \partial_{t}\rho, \mathbb{P}(\partial_{t}v)\Bigg)^{T} = \Bigg(\frac{f(\textbf{u})}{\rho} \partial_{t}\rho, \partial_{t}v \Bigg)^{T},$$ thanks to the divergence free property of the unknown $v$ in (\ref{Goal_System}). This observation leads to the alternative formulation

\begin{equation}
\label{Proj_Formulation}
\Bigg(\frac{f(\textbf{u})}{\rho}\partial_{t}\rho, \partial_{t}v \Bigg)^{T} + \sum_{j=1}^{d} \textbf{P}A_{0}A_{j}(\textbf{u})\partial_{x_{j}}\textbf{u}=0.
\end{equation}

Alternatively, it is possible to project system (\ref{Goal_Compact}) and, after that, apply the symmetrizer $A_{0}$ just to get energy estimates. Thus, we have

\begin{equation}
\label{Proj_Formulation_no_Symm}
\partial_{t}\textbf{u} + \sum_{j=1}^{d} \textbf{P}A_{j}(\textbf{u})\partial_{x_{j}}\textbf{u}=0,
\end{equation}

without any condition on the divergence of the velocity field $v$, i.e. the second equation in (\ref{Goal_Compact}), which is implicitly contained in (\ref{Proj_Formulation}) and (\ref{Proj_Formulation_no_Symm}).

\begin{proposition}
Systems (\ref{Proj_Formulation}) and (\ref{Proj_Formulation_no_Symm}) are equivalent.
\end{proposition}

\begin{proof}
Applying the symmetrizer $A_{0}$ in (\ref{Symmetrizer1}) to equation (\ref{Proj_Formulation_no_Symm}), we get
\begin{equation}
A_{0}(\textbf{u})\partial_{t}\textbf{u}+\sum_{j=1}^{d}A_{0}(\textbf{u})\textbf{P}A_{j}(\textbf{u})\partial_{x_{j}}\textbf{u}=0.
\end{equation}
This way, we have
\begin{equation}
\label{quasi}
A_{0}(\textbf{u})\partial_{t}\textbf{u} + \sum_{j=1}^{d} \textbf{P}A_{0}A_{j}(\textbf{u})\partial_{x_{j}}\textbf{u}+[A_{0}(\textbf{u}), \textbf{P}]A_{j}(\textbf{u})\partial_{x_{j}}\textbf{u}=0.
\end{equation}
Finally, we notice that

\begin{equation}
[A_{0}(\textbf{u}), \textbf{P}]=\Bigg[\left( \begin{array}{cccc}
\frac{f(\textbf{u})}{\rho} & 0 \\
0 &  Id \\
\end{array} \right), \left( \begin{array}{cccc}
1 & 0 \\
0 &  \mathbb{P} \\
\end{array} \right) \Bigg] = 0.
\end{equation}
\end{proof}

Now, we give the definition of classical local solutions to system (\ref{Goal_System}) with initial data (\ref{initial}). 

\begin{definition}
\label{classical_solution}
Let $m > [d/2]+1$ be fixed. The function $\textbf{u}=(\rho, v)$ is a classical solution to system (\ref{Goal_System}) with data (\ref{initial}) if, fixed a positive constant value $\bar{\rho}$, we have that $(\rho-\bar{\rho}, v) \in C([0,T], H^{m}(\mathbb{R}^{d})) \cap C^{1}([0,T], H^{m-1}(\mathbb{R}^{d}))$, with $\rho > 0$, and $\textbf{u}$ solves system

\begin{equation}
\label{Solution}
\begin{cases}
& \partial_{t}\textbf{u} + \sum_{j=1}^{d}A_{j}(\textbf{u})\partial_{x_{j}}\textbf{u}+(0, \nabla P)^{T}=0, \\
& \nabla \cdot v = 0, \\
& \rho(0,x)=\rho_{0}(x), ~~~~~~~~~~ v(0,x)=v_{0}(x) ~~\text{with}~~\nabla \cdot v_{0}(x)=0,
\end{cases}
\end{equation}
where the incompressible pressure $P \in C([0,T], H^{m}(\mathbb{R}^{d}))$. 
\end{definition}

Before proceeding in our proof of well-posedness, we make a slight modification of system (\ref{Symmetric}): fixed a constant value $\bar{\rho}$, we translate the density defining 

\begin{equation}
\label{translation_rho}
\tilde{\rho}:=\rho-\bar{\rho},
\end{equation}

because of reasons discussed in Remark \ref{Translation}, namely (\ref{Symmetrizer1}) and (\ref{Symmetric}) are not defined in $\rho= 0$, then the unknown $\rho$ does not belong to $L^{2}(\mathbb{R}^{d})$, while the translated variable $\rho-\bar{\rho}$ does. We also define

\begin{equation}
\label{translation_u}
\tilde{\textbf{u}}:={\textbf{u}}-\bar{\textbf{u}}=(\tilde{\rho}, \tilde{v})^{T}=(\tilde{\rho}, {v})^{T},
\end{equation}

where $\bar{\textbf{u}}=(\bar{\rho}, 0)$. By this change of variable, system (\ref{Goal_Compact}) becomes

\begin{equation}
\label{Goal_System_Translated}
\partial_{t}\tilde{\textbf{u}}+\sum_{j=1}^{d}A_{j}(\tilde{\textbf{u}}+\bar{\textbf{u}})\partial_{x_{j}}\tilde{\textbf{u}} + F_{P} = 0,
\end{equation}

with initial data $$\tilde{\textbf{u}}_{0}=(\tilde{\rho}_{0}, \tilde{v}_{0})^{T},$$ where

\begin{equation}
\label{translated_initial_data}
\tilde{\rho}_{0}(x)=\rho_{0}(x)-\bar{\rho}, ~~~~~ \tilde{v}_{0}(x)=v_{0}(x),
\end{equation}

and $\rho_{0}, v_{0}$ are the original initial data in (\ref{initial}).

\subsection{Uniqueness}
We end this section with the proof of uniqueness of solutions to system (\ref{Goal_System_Translated}). According to Definition \ref{classical_solution}, let $\tilde{\textbf{u}}_{1}, \tilde{\textbf{u}}_{2}$ be two solutions to system (\ref{Goal_System_Translated}) with initial data (\ref{translated_initial_data}). We can write

\begin{equation*}
A_{0}(\tilde{\textbf{u}}_{2}+\bar{\textbf{u}})\partial_{t}(\tilde{\textbf{u}}_{2}-\tilde{\textbf{u}}_{1})+\sum_{j=1}^{d}A_{0}A_{j}(\tilde{\textbf{u}}_{2}+\bar{\textbf{u}})\partial_{x_{j}}(\tilde{\textbf{u}}_{2}-\tilde{\textbf{u}}_{1}) + \left( \begin{array}{c}
0 \\
 \nabla P_{2} - \nabla P_{1} \\
\end{array} \right)
\end{equation*}

\begin{equation}
=(A_{0}(\tilde{\textbf{u}}_{1}+\bar{\textbf{u}})-A_{0}(\tilde{\textbf{u}}_{2}+\bar{\textbf{u}}))\partial_{t}{\textbf{u}}_{1}+\sum_{j=1}^{d}(A_{0}A_{j}(\tilde{\textbf{u}}_{1}+\bar{\textbf{u}})-A_{0}A_{j}(\tilde{\textbf{u}}_{2}+\bar{\textbf{u}}))\partial_{x_{j}}\tilde{\textbf{u}}_{1}.
\end{equation}

Setting $\textbf{w}:=\tilde{\textbf{u}}_{2}-\tilde{\textbf{u}}_{1}$ and taking the scalar product against $\textbf{w}$, we obtain

\begin{equation}
\label{uniqueness}
\frac{d}{dt}\int_{\mathbb{R}^{d}} A_{0}(\tilde{\textbf{u}}_{2}+\bar{\textbf{u}}) \textbf{w} \cdot \textbf{w} ~ dx \le c \int_{\mathbb{R}^{d}} A_{0}(\tilde{\textbf{u}}_{2}+\bar{\textbf{u}})  \textbf{w} \cdot \textbf{w} ~ dx,
\end{equation}

i.e. $\textbf{w}=0$, since $\tilde{\textbf{u}}_{1}(0,x)=\tilde{\textbf{u}}_{2}(0,x)=\tilde{\textbf{u}}_{0}=(\tilde{\rho}_{0}(x), \tilde{v}_{0}(x))^{T}$ in (\ref{translated_initial_data}).

\begin{remark}
We point out the fact that, in order to have uniqueness of solutions to system (\ref{Goal_System_Translated}), it is enough to require $\tilde{\textbf{u}} \in Lip([0, T], Lip(\mathbb{R}^{d})) \cap L^{\infty}([0,T], L^{2}(\mathbb{R}^{d}))$, because of the fact that the constant $c$ in (\ref{uniqueness}) depends on $|\tilde{\textbf{u}}|_{\infty}$, $|\partial_{t}\tilde{\textbf{u}}|_{\infty}$ and $|\nabla \tilde{\textbf{u}}|_{\infty}$.
\end{remark}

\section{Well-posedness via paradifferential calculus}
Following \cite{Bertozzi}, the first idea is to approximate the compact and translated version (\ref{Goal_System_Translated}) of system (\ref{Goal_Compact}) by a simple regularization of the equations, using mollifiers $J_{\varepsilon}$. 
\begin{definition}
\label{mollifiers}
Let $\Phi(|x|) \in C_{0}^{\infty}(\mathbb{R}^{d})$ be any positive, radial function such that $\int_{\mathbb{R}^{d}} \Phi ~ dx =1$. Fix $\varepsilon > 0$, and let $j_{\varepsilon}=\Phi(x/\varepsilon)$, the mollification $J_{\varepsilon} \textbf{w}$ of functions $\textbf{w} \in L^{2}(\mathbb{R}^{d})$ is defined by

\begin{equation}
\label{def_mollifiers}
J_{\varepsilon}\textbf{w}(x) =(j_{\varepsilon} * \textbf{w})(x) = \int_{\mathbb{R}^{d}} \Phi\Bigg(\frac{x-y}{\varepsilon}\Bigg)\textbf{w}(y) ~ dy.
\end{equation}
\end{definition}

\begin{lemma}
\label{commutation}
The modified version $\textbf{P}$ of the Leray projector $\mathbb{P}$ commutes with the diagonal matrix with mollifiers $J_{\varepsilon}$ entries.
\end{lemma}
\begin{proof}
By definition (\ref{def_mollifiers}) and Fourier transform property, it follows that
\begin{equation}
J_{\varepsilon}\textbf{w}(x)=(j_{\varepsilon}*\textbf{w})(x)=\mathcal{F}^{-1}(\hat{{j}_{\varepsilon}}(\xi) \hat{\textbf{w}}(\xi)).
\end{equation}
This implies that $\hat{{j}_{\varepsilon}}(\xi)$, the symbol of the mollifier operator $J_{\varepsilon}$, only depends on $\xi$, as also the projector $\textbf{P}$ in (\ref{Projector}). Then, we are dealing with two different Fourier multipliers, which commute (see \cite{Metivier}).
\end{proof}

Now, recalling equations (\ref{Goal_System}) and since the divergence free condition for the velocity field $v$ holds, we can reduce the first equation of system (\ref{Goal_System}) to a transport equation, obtaining the following equations:

\begin{equation}
\label{Goal_System_Reduced}
\begin{cases}
& \partial_{t}\rho + v \cdot \nabla \rho=0, \\
& \partial_{t}{v}+{v} \cdot \nabla {v} + f(\rho, v) \nabla \rho+ \nabla {P}=0, \\
& \nabla \cdot v = 0.
\end{cases}
\end{equation}

This observation leads to a slight modification, compared to (\ref{A11}) and (\ref{Ajj}), in the expressions of the matrices $A_{j}(\textbf{u})$ associated to the compact formulation $$\displaystyle \partial_{t}\textbf{u}+\sum_{j=1}^{d}A_{j}(\textbf{u})\partial_{x_{j}}\textbf{u}=0$$ of system (\ref{Goal_System_Reduced}). Then, we have
\begin{equation}
\label{A111}
A_{1}(\textbf{u}) =  \left( \begin{array}{ccc}
v_{1} & 0 & 0 \\
f(\textbf{u}) & v_{1} & 0 \\
0 & 0 & v_{1} 
\end{array} \right), \ 
A_{2}(\textbf{u}) =  \left( \begin{array}{ccc}
v_{2} & 0 & 0 \\
0 & v_{2} & 0 \\
f(\textbf{u})  & 0 & v_{2} 
\end{array} \right),
\end{equation}

in the $2$-dimensional case and, in the general $d$-case, for $j=1, \cdots, d,$

\begin{equation}
\label{Ajjj}
A_{j}(\textbf{u}) =  \left( \begin{array}{ccccc}
v_{j} &0  & 0 & \cdots & 0\\
\delta_{j1} f(\textbf{u}) & v_{j} & 0 & \cdots & 0 \\
\delta_{j2} f(\textbf{u}) & 0 & v_{j} & \cdots & 0 \\
\cdots &  \cdots & \cdots &  v_{j} & \cdots \\
\delta_{jd} f(\textbf{u}) & 0 & 0 & \cdots & v_{j} \\
\end{array} \right).
\end{equation}

Taking into account the translation and the related definition of the new variable $\tilde{\textbf{u}}$ that we have made in the previous section, we are ready to define our approximating system of equations (\ref{Goal_System_Translated}):

\begin{equation}
\label{P_approx}
\begin{cases}
& \partial_{t} \tilde{\textbf{u}}^{\varepsilon} + \sum_{j=1}^{d} J_{\varepsilon}A_{j}(J_{\varepsilon}(\tilde{\textbf{u}}^{\varepsilon}+\bar{\textbf{u}}))\partial_{x_{j}}J_{\varepsilon}\tilde{\textbf{u}}^{\varepsilon} + (0, \nabla P^{\varepsilon})^{T}=0, \\
& \nabla \cdot \tilde{v}^{\varepsilon}=0,
\end{cases}
\end{equation}

with initial data $\tilde{\textbf{u}}^{\varepsilon}_{0}(x) = \tilde{\textbf{u}}_{0}(x)=(\tilde{\rho}_{0}(x), \tilde{v}_{0}(x))^{T}$ in (\ref{translated_initial_data}). Notice that this approximation explicitly contains the unknown pressure $P^{\varepsilon}$. Now, the idea is to eliminate that pressure term, by applying to (\ref{P_approx}) the modified version (\ref{Projector}) of the Leray projector operator. Recalling that $\textbf{P}$ in (\ref{Projector}) projects any vector $\tilde{\textbf{u}}=(\tilde{\rho}, \tilde{v})$ onto the space 

\begin{equation}
\label{Vs}
V^{s}:=\{(\tilde{\rho}, \tilde{v}) \in H^{s}(\mathbb{R}^{d}) : \nabla \cdot \tilde{v}=0\},
\end{equation}
and, since $\nabla \cdot v^{\varepsilon} = 0$, i.e. $\tilde{\textbf{u}}^{\varepsilon} \in V^{s},$  it follows that $$\textbf{P}\tilde{\textbf{u}}^{\varepsilon}=(\tilde{\rho}^{\varepsilon}, \mathbb{P}\tilde{v}^{\varepsilon})^{T} = (\tilde{\rho}^{\varepsilon}, \tilde{v}^{\varepsilon})^{T}=\tilde{\textbf{u}}^{\varepsilon}.$$ Equalities $\textbf{P}\partial_{t}\tilde{\textbf{u}}^{\varepsilon}=\partial_{t}\textbf{P}\tilde{\textbf{u}}^{\varepsilon}=\partial_{t}\tilde{\textbf{u}}^{\varepsilon}$ lead to another version of system (\ref{P_approx}):

\begin{equation}
\label{P_approx_compact}
\partial_{t} \tilde{\textbf{u}}^{\varepsilon} + \sum_{j=1}^{d} \textbf{P} J_{\varepsilon}A_{j}(J_{\varepsilon}(\tilde{\textbf{u}}^{\varepsilon}+\bar{\textbf{u}}))\partial_{x_{j}}J_{\varepsilon}\tilde{\textbf{u}}^{\varepsilon}=0,
\end{equation}

with, again, initial data $\tilde{\textbf{u}}^{\varepsilon}_{0}(x) = \tilde{\textbf{u}}_{0}(x)$. Now, since the mollifiers properties hold, we are able to use the Picard theorem on Banach spaces to get local solutions to system (\ref{P_approx_compact}).

\begin{theorem} (Picard Theorem on a Banach Space, \cite{Bertozzi}).
\label{Picard}
Let $\mathcal{U} \subseteq \textbf{B}$ be an open subset of a Banach space $\textbf{B}$ and let $F: \mathcal{U} \rightarrow \textbf{B}$ be a mapping that satisfies the following hypothesis:

\begin{itemize}
\item $F(x)$ maps $\mathcal{U}$ to $\textbf{B}$. 
\item $F$ is locally Lipschitz continuous, i.e., for any $x \in \mathcal{U}$ there exist $L>0$ and an open neighbourhood $\mathcal{U}_{x}$ of $x$ s. t.
\end{itemize}

\begin{equation}
||F(x_{1})-F(x_{2})||_{\textbf{B}} \le L ||x_{1}-x_{2}||_{\textbf{B}}
\end{equation}

for all $x_{1}, x_{2} \in \mathcal{U}_{x}$. Then, for any $x_{0} \in \mathcal{U}$, there exists a time $T$ s.t. the ODE

\begin{equation}
\frac{d x}{dt} = F(x), ~~~~~ x|_{t=0} = x_{0} \in \mathcal{U},
\end{equation}

has a unique local solution $x \in C^{1}((-T, T), \mathcal{U})$.
\end{theorem}

System (\ref{P_approx_compact}) reduces to an ordinary differential equation:

\begin{equation}
\label{ODE_P}
\partial_{t}\tilde{\textbf{u}}^{\varepsilon} = F^{\varepsilon}(\tilde{\textbf{u}}^{\varepsilon}),  ~~~~~~~~~~\tilde{\textbf{u}}^{\varepsilon}|_{t=0}=\tilde{\textbf{u}}^{\varepsilon}_{0}(x),
\end{equation}

where 

\begin{equation}
\label{Function_ODE_P}
F^{\varepsilon}(\tilde{\textbf{u}}^{\varepsilon}) = -\sum_{j=1}^{d}\textbf{P} J_{\varepsilon}A_{j}(J_{\varepsilon}(\tilde{\textbf{u}^{\varepsilon}}+\bar{\textbf{u}}))\partial_{x_{j}}J_{\varepsilon}\tilde{\textbf{u}}^{\varepsilon}.
\end{equation}

We prove the following theorem:

\begin{theorem}(Local existence of approximating solutions for the first type of approximation)
\label{Local Existence of Approximating Solutions_P}
Let $\tilde{\textbf{u}}^{\varepsilon}_{0}=(\tilde{\rho}_{0}^{\varepsilon}, \tilde{v}_{0}^{\varepsilon})^{T}$ as in (\ref{translated_initial_data}) belonging to $V^{s}$ defined in (\ref{Vs}), with $s >  d/2+1$. Then, for any $\varepsilon > 0$, there exists a time $T$, independent of $\varepsilon$, such that system (\ref{P_approx_compact}) has a unique solution $\tilde{\textbf{u}}^{\varepsilon}=(\tilde{\rho}^{\varepsilon}, \tilde{v}^{\varepsilon})^{T} \in C^{1}([0, T], V^{s})$.
\end{theorem}

\begin{proof}
First of all, we show that existence and uniqueness follow from the Picard theorem. Then, we prove that the time of local existence $ T_{\varepsilon}$ can be bounded from below by any time $T>0$, which is independent of $\varepsilon$. We point out that $J_{\varepsilon}\tilde{\textbf{u}}^{\varepsilon}$ and $J_{\varepsilon}(\tilde{\textbf{u}}^{\varepsilon}+\bar{\textbf{u}})$ are $C^{\infty}$ functions and, recalling \cite{Metivier}, $\mathbb{P}$ is associated to an analytic pseudodifferential operator of order 0, modulo an infinitely smooth remainder, therefore $$F^{\varepsilon}: V^{s} \rightarrow V^{s},$$ because of the fact that $\nabla \cdot \tilde{v}^{\varepsilon}=0$ in (\ref{P_approx_compact}) and, explicitly, in (\ref{P_approx}). In order to apply the Picard theorem, we have to prove that $F^{\varepsilon}(\tilde{\textbf{u}}^{\varepsilon})$ in (\ref{Function_ODE_P}) is Lispchitz continuous. To do this, we take two vectors $\tilde{\textbf{u}}_{1}, \tilde{\textbf{u}}_{2}$ in the unknowns' space. In the following, we omit the index $\varepsilon$ in the unknown functions, where there is no ambiguity. Let $c_{S}$ be the Sobolev embedding constant. Therefore

\begin{equation*}
||F^{\varepsilon}(\tilde{\textbf{u}}_{1})-F^{ \varepsilon}(\tilde{\textbf{u}}_{2})||_{0} 
\end{equation*}

\begin{equation*}
\le \sum_{j=1}^{d}||\textbf{P} J_{\varepsilon}A_{j}(J_{\varepsilon}(\tilde{\textbf{u}}_{1}+\bar{\textbf{u}}))\partial_{x_{j}}J_{\varepsilon}\tilde{\textbf{u}}_{1} - \textbf{P}J_{\varepsilon}A_{j}(J_{\varepsilon}(\tilde{\textbf{u}}_{2}+\bar{\textbf{u}}))\partial_{x_{j}}J_{\varepsilon}\tilde{\textbf{u}}_{2}||_{0}
\end{equation*}

\begin{equation*}
\le  \sum_{j=1}^{d}\{||[\textbf{P}J_{\varepsilon}A_{j}(J_{\varepsilon}(\tilde{\textbf{u}}_{1}+\bar{\textbf{u}})) - \textbf{P}J_{\varepsilon}A_{j}(J_{\varepsilon}(\tilde{\textbf{u}}_{2}+\bar{\textbf{u}}))]\partial_{x_{j}}J_{\varepsilon}\tilde{\textbf{u}}_{1}||_{0}
\end{equation*}

\begin{equation*}
+ ||\textbf{P}J_{\varepsilon}A_{j}(J_{\varepsilon}(\tilde{\textbf{u}}_{2}+\bar{\textbf{u}}))\partial_{x_{j}}J_{\varepsilon}(\tilde{\textbf{u}}_{1}-\tilde{\textbf{u}}_{2})||_{0}\}
\end{equation*}

\begin{equation*}
= \sum_{j=1}^{d} \Bigg\{\Bigg|\Bigg|  \textbf{P}\Bigg[ \int_{0}^{1}\frac{d}{dr}(J_{\varepsilon}A_{j}(r J_{\varepsilon}(\tilde{\textbf{u}}_{1}+\bar{\textbf{u}}) + (1-r) J_{\varepsilon} (\tilde{\textbf{u}}_{2}+\bar{\textbf{u}}))) ~ dr \Bigg] \partial_{x_{j}}J_{\varepsilon}\tilde{\textbf{u}}_{1} \Bigg|\Bigg|_{0}
\end{equation*}

\begin{equation*}
+ c(|J_{\varepsilon}(\tilde{\textbf{u}}_{2}+\bar{\textbf{u}})|_{\infty}) ||\partial_{x_{j}}J_{\varepsilon}(\tilde{\textbf{u}}_{1}-\tilde{\textbf{u}}_{2})||_{0} \Bigg\}
\end{equation*}

\begin{equation*}
= \sum_{j=1}^{d} \Bigg\{\Bigg|\Bigg| \textbf{P}\Bigg[ \int_{0}^{1} (J_{\varepsilon}A_{j}(r J_{\varepsilon}(\tilde{\textbf{u}}_{1}+\bar{\textbf{u}}) + (1-r) J_{\varepsilon} (\tilde{\textbf{u}}_{2}+\bar{\textbf{u}})))' dr \Bigg]J_{\varepsilon}(\tilde{\textbf{u}}_{1}-\tilde{\textbf{u}}_{2}) ~\partial_{x_{j}}J_{\varepsilon}\tilde{\textbf{u}}_{1} \Bigg|\Bigg|_{0}
\end{equation*}

\begin{equation*}
+ c(|J_{\varepsilon}(\tilde{\textbf{u}}_{2}+\bar{\textbf{u}})|_{\infty}) ||\partial_{x_{j}}J_{\varepsilon}(\tilde{\textbf{u}}_{1}-\tilde{\textbf{u}}_{2})||_{0} \Bigg\}
\end{equation*}

\begin{equation*}
\le  \sum_{j=1}^{d} \{c(|J_{\varepsilon}(\tilde{\textbf{u}}_{1}+\bar{\textbf{u}})|_{\infty}, |J_{\varepsilon}(\tilde{\textbf{u}}_{2}+\bar{\textbf{u}})|_{\infty}, |\partial_{x_{j}}J_{\varepsilon}\tilde{\textbf{u}}_{1}|_{\infty}) (||J_{\varepsilon}(\tilde{\textbf{u}}_{1}-\tilde{\textbf{u}}_{2})||_{0} +||\partial_{x_{j}}J_{\varepsilon}(\tilde{\textbf{u}}_{1}-\tilde{\textbf{u}}_{2})||_{0})\}
\end{equation*}

\begin{equation*}
\le c(|J_{\varepsilon}(\tilde{\textbf{u}}_{1}+\bar{\textbf{u}})|_{\infty}, |J_{\varepsilon}(\tilde{\textbf{u}}_{2}+\bar{\textbf{u}})|_{\infty}, |\nabla J_{\varepsilon}\tilde{\textbf{u}}_{1}|_{\infty}) ||J_{\varepsilon}(\tilde{\textbf{u}}_{1}-\tilde{\textbf{u}}_{2})||_{1}
\end{equation*}

\begin{equation}
\le c(c_{S}, ||\tilde{\textbf{u}}_{1}||_{s}, ||\tilde{\textbf{u}}_{2}||_{s}, \bar{\rho}) ||\tilde{\textbf{u}}_{1}-\tilde{\textbf{u}}_{2}||_{1},
\end{equation}

with $\bar{\rho}$ in (\ref{translation_rho}) and where last inequality follows from Moser estimates and properties of mollifiers. Taking the $\alpha$ ($|\alpha| \le s$) derivative and the using the commutation property of $J_{\varepsilon}$ and $\textbf{P}$ in  Lemma \ref{commutation}, we have 

\begin{equation*}
||D^{\alpha}(F^{\varepsilon}(\tilde{\textbf{u}}_{1})-F^{\varepsilon}(\tilde{\textbf{u}}_{2}))||_{0} 
\end{equation*}

\begin{equation*}
\le \sum_{j=1}^{d}||D^{\alpha}(\textbf{P} J_{\varepsilon}A_{j}(J_{\varepsilon}(\tilde{\textbf{u}}_{1}+\bar{\textbf{u}}))\partial_{x_{j}}J_{\varepsilon}\tilde{\textbf{u}}_{1} - \textbf{P} J_{\varepsilon}A_{j}(J_{\varepsilon}(\tilde{\textbf{u}}_{2}+\bar{\textbf{u}}))\partial_{x_{j}}J_{\varepsilon}\tilde{\textbf{u}}_{2})||_{0}
\end{equation*}

\begin{equation*}
\le \sum_{j=1}^{d}\{|| \textbf{P}D^{\alpha}[(J_{\varepsilon}A_{j}(J_{\varepsilon}(\tilde{\textbf{u}}_{1}+\bar{\textbf{u}})) -J_{\varepsilon}A_{j}(J_{\varepsilon}(\tilde{\textbf{u}}_{2}+\bar{\textbf{u}})))\partial_{x_{j}}J_{\varepsilon}\tilde{\textbf{u}}_{1}]||_{0}
\end{equation*}

\begin{equation*}
+ ||\textbf{P}D^{\alpha}(J_{\varepsilon}A_{j}(J_{\varepsilon}(\tilde{\textbf{u}}_{2}+\bar{\textbf{u}}))\partial_{x_{j}}J_{\varepsilon}(\tilde{\textbf{u}}_{1}-\tilde{\textbf{u}}_{2}))||_{0}\}
\end{equation*}

\begin{equation*}
\le c_{S} \sum_{j=1}^{d}\{||\textbf{P}D^{s}(J_{\varepsilon}A_{j}(J_{\varepsilon}(\tilde{\textbf{u}}_{1}+\bar{\textbf{u}})) - J_{\varepsilon}A_{j}(J_{\varepsilon}(\tilde{\textbf{u}}_{2}+\bar{\textbf{u}})))||_{0} |\partial_{x_{j}}J_{\varepsilon}\tilde{\textbf{u}}_{1}|_{\infty}
\end{equation*}

\begin{equation*}
+ |\textbf{P}[J_{\varepsilon}A_{j}(J_{\varepsilon}(\tilde{\textbf{u}}_{1}+\bar{\textbf{u}})) - J_{\varepsilon}A_{j}(J_{\varepsilon}(\tilde{\textbf{u}}_{2}+\bar{\textbf{u}}))]|_{\infty} ||D^{s}\partial_{x_{j}}J_{\varepsilon}\tilde{\textbf{u}}_{1}||_{0}
\end{equation*}

\begin{equation*}
+||\textbf{P}D^{s}(J_{\varepsilon}A_{j}(J_{\varepsilon}(\tilde{\textbf{u}}_{2}+\bar{\textbf{u}})))||_{0} |\partial_{x_{j}}J_{\varepsilon}(\tilde{\textbf{u}}_{1}-\tilde{\textbf{u}}_{2})|_{\infty}
\end{equation*}

\begin{equation*}
+|\textbf{P}J_{\varepsilon}A_{j}(J_{\varepsilon}(\tilde{\textbf{u}}_{2}+\bar{\textbf{u}}))|_{\infty} ||D^{s}\partial_{x_{j}}J_{\varepsilon}(\tilde{\textbf{u}}_{1}-\tilde{\textbf{u}}_{2})||_{0}\}
\end{equation*}

\begin{equation*}
= c_{S}\sum_{j=1}^{d}\Bigg\{\Bigg|\Bigg|\textbf{P} D^{s}\Bigg[\int_{0}^{1} \frac{d}{dr}(J_{\varepsilon}A_{j}(r J_{\varepsilon}(\tilde{\textbf{u}}_{1}+\bar{\textbf{u}}) + (1-r) J_{\varepsilon} (\tilde{\textbf{u}}_{2}+\bar{\textbf{u}}))) ~ dr\Bigg]\Bigg|\Bigg|_{0} |\partial_{x_{j}}J_{\varepsilon}\tilde{\textbf{u}}_{1}|_{\infty}
\end{equation*}

\begin{equation*}
+ \Bigg|\textbf{P}\int_{0}^{1} \frac{d}{dr}(J_{\varepsilon}A_{j}(r J_{\varepsilon}(\tilde{\textbf{u}}_{1}+\bar{\textbf{u}}) + (1-r) J_{\varepsilon} (\tilde{\textbf{u}}_{2}+\bar{\textbf{u}}))) ~ dr\Bigg|_{\infty}  ||D^{s}\partial_{x_{j}}J_{\varepsilon}\tilde{\textbf{u}}_{1}||_{0}\Bigg\}
\end{equation*}

\begin{equation*}
+||\textbf{P}D^{s}(J_{\varepsilon}A_{j}(J_{\varepsilon}(\tilde{\textbf{u}}_{2}+\bar{\textbf{u}})))||_{0} |\partial_{x_{j}}J_{\varepsilon}(\tilde{\textbf{u}}_{1}-\tilde{\textbf{u}}_{2})|_{\infty}
\end{equation*}

\begin{equation*}
+|\textbf{P}J_{\varepsilon}A_{j}(J_{\varepsilon}(\tilde{\textbf{u}}_{2}+\bar{\textbf{u}}))|_{\infty} ||D^{s}\partial_{x_{j}}J_{\varepsilon}(\tilde{\textbf{u}}_{1}-\tilde{\textbf{u}}_{2})||_{0}\}
\end{equation*}

\begin{equation*}
=c_{S}\sum_{j=1}^{d}\Bigg\{\Bigg|\Bigg|\textbf{P}D^{s}\Bigg[\int_{0}^{1} dr (J_{\varepsilon}A_{j}(r J_{\varepsilon}(\tilde{\textbf{u}}_{1}+\bar{\textbf{u}}) + (1-r) J_{\varepsilon} (\tilde{\textbf{u}}_{2}+\bar{\textbf{u}})))'   J_{\varepsilon}(\tilde{\textbf{u}}_{1}-\tilde{\textbf{u}}_{2})\Bigg]\Bigg|\Bigg|_{0} 
\end{equation*}

\begin{equation*}
|\partial_{x_{j}}J_{\varepsilon}\tilde{\textbf{u}}_{1}|_{\infty}
\end{equation*}

\begin{equation*}
+ \Bigg|\textbf{P} \int_{0}^{1} dr (J_{\varepsilon}A_{j}(r J_{\varepsilon}(\tilde{\textbf{u}}_{1}+\bar{\textbf{u}}) + (1-r) J_{\varepsilon} (\tilde{\textbf{u}}_{2}+\bar{\textbf{u}})))'  J_{\varepsilon}(\tilde{\textbf{u}}_{1}-\tilde{\textbf{u}}_{2}) \Bigg|_{\infty} 
\end{equation*}

\begin{equation*}
||D^{s}\partial_{x_{j}}J_{\varepsilon}\tilde{\textbf{u}}_{1}||_{0}\Bigg\}
\end{equation*}

\begin{equation*}
+||\textbf{P}D^{s}(J_{\varepsilon}A_{j}(J_{\varepsilon}(\tilde{\textbf{u}}_{2}+\bar{\textbf{u}})))||_{0} |\partial_{x_{j}}J_{\varepsilon}(\tilde{\textbf{u}}_{1}-\tilde{\textbf{u}}_{2})|_{\infty}
\end{equation*}

\begin{equation*}
+|\textbf{P}J_{\varepsilon}A_{j}(J_{\varepsilon}(\tilde{\textbf{u}}_{2}+\bar{\textbf{u}}))|_{\infty} ||D^{s}\partial_{x_{j}}J_{\varepsilon}(\tilde{\textbf{u}}_{1}-\tilde{\textbf{u}}_{2})||_{0}\}
\end{equation*}

and, then,

\begin{equation*}
\le c(c_{S}, ||\tilde{\textbf{u}}_{1}||_{m}, ||\tilde{\textbf{u}}_{2}||_{m}, \bar{\rho}, \varepsilon^{-1}) ||D^{s}(\tilde{\textbf{u}}_{1}-\tilde{\textbf{u}}_{2})||_{0},
\end{equation*}

\begin{equation}
\label{last_first_appr}
= c(c_{S}, ||\tilde{\textbf{u}}_{1}||_{m}, ||\tilde{\textbf{u}}_{2}||_{m}, \bar{\rho}, \varepsilon^{-1}) ||\tilde{\textbf{u}}_{1}-\tilde{\textbf{u}}_{2}||_{s},
\end{equation}

where, once again, last inequality follows from Moser estimates and properties of mollifiers, as we can see in the following remark.

\begin{remark}
\label{m_estimate}
Define $G(\tilde{\textbf{u}}_{2}):=\textbf{P}J_{\varepsilon} [A_{j}(J_{\varepsilon}(\tilde{\textbf{u}}_{2}+\bar{\textbf{u}}))-A_{j}(\bar{\textbf{u}})]$. It follows that $G(\underline{0})=0$. Applying Theorem C.12 in \cite{Benzoni}, we have that $$||G(\tilde{\textbf{u}}_{2})||_{s} = ||\textbf{P}J_{\varepsilon} [A_{j}(J_{\varepsilon}(\tilde{\textbf{u}}_{2}+\bar{\textbf{u}}))-A_{j}(\bar{\textbf{u}})]||_{s} \le  ||A_{j}(J_{\varepsilon}(\tilde{\textbf{u}}_{2}+\bar{\textbf{u}}))-A_{j}(\bar{\textbf{u}})||_{s}$$ $$\le c(|\tilde{\textbf{u}}_{2}|_{\infty})||\tilde{\textbf{u}}_{2}||_{s}.$$ Therefore
$$ ||\textbf{P} D^{s}(J_{\varepsilon}A_{j}(J_{\varepsilon}(\tilde{\textbf{u}}_{2}+\bar{\textbf{u}})))||_{0}$$

$$ =||\textbf{P} D^{s}(J_{\varepsilon}A_{j}(J_{\varepsilon}(\tilde{\textbf{u}}_{2}+\bar{\textbf{u}}))-\ J_{\varepsilon}A_{j}(\bar{\textbf{u}})) + \textbf{P} D^{s}( J_{\varepsilon}A_{j}(\bar{\textbf{u}}))||_{0} $$

$$ = ||\textbf{P} [D^{s}(J_{\varepsilon}A_{j}(J_{\varepsilon}(\tilde{\textbf{u}}_{2}+\bar{\textbf{u}}))- J_{\varepsilon}A_{j}(\bar{\textbf{u}}))]||_{0} $$

$$ \le c(|\tilde{\textbf{u}}_{2}|_{\infty})||\tilde{\textbf{u}}_{2}||_{s} \le c(c_{S}||\tilde{\textbf{u}}_{2}||_{s})||\tilde{\textbf{u}}_{2}||_{s}. $$
\end{remark}

Last inequality (\ref{last_first_appr}) implies that, for fixed $\varepsilon$, $F^{\varepsilon}$ is locally Lipschitz continuous on any open set

\begin{equation}
\mathcal{U}^{M}=\{\tilde{\textbf{u}}^{\varepsilon} \in V^{s} : ||\tilde{\textbf{u}}^{\varepsilon}||_{s} \le M \}.
\end{equation}

By using the Picard theorem, there exists a unique solution $\tilde{\textbf{u}}^{ \varepsilon} \in C^{1}([0, T_{\varepsilon}), \mathcal{U}^{M})$ for any $T_{\varepsilon} > 0$.

\par\bigskip

Now, we want to show that the time of existence $T_{\varepsilon}$ is bounded from below by any strictly positive time $T$ that is independent of $\varepsilon$. Roughly speaking, we need a uniform - in $\varepsilon$ - bound on $\tilde{\textbf{u}}^{\varepsilon}$ in the higher $H^{s}$-norm. First of all, system (\ref{P_approx_compact}), i.e.

$$\partial_{t}\tilde{\textbf{u}}^{\varepsilon}+\sum_{j=1}^{d}\textbf{P}J_{\varepsilon}A_{j}(J_{\varepsilon}(\tilde{\textbf{u}}^{\varepsilon}+\bar{\textbf{u}}))\partial_{x_{j}}J_{\varepsilon}\tilde{\textbf{u}}^{\varepsilon}=0,$$ by the commutation property of $J_{\varepsilon}$ and $\textbf{P}$ in Lemma \ref{commutation}, can be written as

\begin{equation}
\label{P_last}
\partial_{t}\tilde{\textbf{u}}^{\varepsilon}+\sum_{j=1}^{d}J_{\varepsilon}\textbf{P}A_{j}(J_{\varepsilon}(\tilde{\textbf{u}}^{\varepsilon}+\bar{\textbf{u}}))\partial_{x_{j}}\tilde{\textbf{u}}^{\varepsilon}=0.
\end{equation}

Let $\textbf{P}$ and $T_{iA}$ be respectively the pseudodifferential and the paradifferential operators associated to the regularized symbol $\textbf{P}(\xi)$ in (\ref{Projector}) and 

$$A:=A(\xi, J_{\varepsilon}(\tilde{\textbf{u}}^{\varepsilon}+\bar{\textbf{u}}))=\sum_{j=1}^{d}  A_{j}(J_{\varepsilon}(\tilde{\textbf{u}}^{\varepsilon}+\bar{\textbf{u}}))\xi_{j}$$

\begin{equation}
\label{A_symbol}
= \left( \begin{array}{ccccc}
\sum_{j=1}^{d} J_{\varepsilon}v_{j}\xi_{j} &0  & 0 & \cdots & 0\\
\xi_{1} f(J_{\varepsilon}(\tilde{\textbf{u}}^{\varepsilon}+\bar{\textbf{u}})) & \sum_{j=1}^{d} J_{\varepsilon}v_{j}\xi_{j} & 0 & \cdots & 0 \\
\xi_{2} f(J_{\varepsilon}(\tilde{\textbf{u}}^{\varepsilon}+\bar{\textbf{u}})) & 0 & \sum_{j=1}^{d}J_{\varepsilon}v_{j}\xi_{j} & \cdots & 0 \\
\cdots &  \cdots & \cdots &  \sum_{j=1}^{d}v_{j}\xi_{j} & \cdots \\
\xi_{d} f(J_{\varepsilon}(\tilde{\textbf{u}}^{\varepsilon}+\bar{\textbf{u}})) & 0 & 0 & \cdots & \sum_{j=1}^{d}J_{\varepsilon}v_{j}\xi_{j} \\
\end{array} \right),
\end{equation}
where matrices $A_{j}$ have been written in (\ref{A111}) and (\ref{Ajjj}). The paradifferential version of system (\ref{P_last}) is the following:
 
\begin{equation}
\label{Paradifferential_system}
\partial_{t}\tilde{\textbf{u}}^{\varepsilon}+J_{\varepsilon}\textbf{P} T_{iA} J_{\varepsilon}\tilde{\textbf{u}}^{\varepsilon} = - \Bigg[\sum_{j=1}^{d}J_{\varepsilon}\textbf{P}A_{j}(J_{\varepsilon}(\tilde{\textbf{u}}^{\varepsilon}+\bar{\textbf{u}}))\partial_{x_{j}}J_{\varepsilon} \tilde{\textbf{u}}^{\varepsilon} - J_{\varepsilon}\textbf{P} T_{iA}J_{\varepsilon} \tilde{\textbf{u}}^{\varepsilon} \Bigg].
\end{equation}

Following \emph{G. M\'etivier}, Lemma 7.2.3 in \cite{Metivier}, let $T_{iA_{j}}$ be the paradifferential operator associated to the symbol $A_{j}(J_{\varepsilon}(\tilde{\textbf{u}}^{\varepsilon}+\bar{\textbf{u}}))\xi_{j}$ for $j=1, \cdots, d$. Then

\begin{equation}
||[T_{iA_{j}} - A_{j}(J_{\varepsilon}(\tilde{\textbf{u}}^{\varepsilon}+\bar{\textbf{u}}))\partial_{x_{j}}]J_{\varepsilon}\tilde{\textbf{u}}^{\varepsilon}||_{s} \le c(||\tilde{\textbf{u}}^{\varepsilon}||_{s}, \bar{\rho}) ||J_{\varepsilon}\tilde{\textbf{u}}^{\varepsilon}||_{s}.
\end{equation}

Therefore, we can focus on the left hand side of equation (\ref{Paradifferential_system}), which is the following paradifferential equation:

\begin{equation}
\label{Paradifferential_reduced}
\partial_{t}\tilde{\textbf{u}}^{\varepsilon}+J_{\varepsilon}\textbf{P} T_{iA} J_{\varepsilon}\tilde{\textbf{u}}^{\varepsilon} = 0.
\end{equation}

At this point, we have to deal with the composition $\textbf{P} \circ T_{iA}$. Following \cite{Metivier} and, to be precise, referring to \emph{E. Grenier}, Proposition 1.10 in  \cite{Grenier}, it is known that the symbol associated to the composition is made by a sum, in the $\alpha$ multi-index, of terms like

$$D_{\xi}^{\alpha}\textbf{P}(\xi) D_{x}^{\alpha}A(\xi, J_{\varepsilon}(\tilde{\textbf{u}}^{\varepsilon}+\bar{\textbf{u}})),$$

where $D_{j}=\frac{1}{i}\partial_{j}$. Apart from $|\alpha|=0$, the others are terms of order less than or equal to 0. This means that the symbol related to the operator $\textbf{P} \circ T_{iA}$, can be written as

\begin{equation}
\label{symbol_composition}
\textbf{P}(\xi)A(\xi, J_{\varepsilon}(\tilde{\textbf{u}}^{\varepsilon}+\bar{\textbf{u}}))+R_{\alpha},
\end{equation}
 
where $R_{\alpha}$ is the remainder of order less than or equal to 0. By (\ref{A_symbol}) and (\ref{Projector}),

$$\textbf{P}A = \textbf{P}(\xi) A(\xi, J_{\varepsilon}(\tilde{\textbf{u}}^{\varepsilon}+\bar{\textbf{u}}))$$

\begin{equation}
\label{PA_symbol}
= \left( \begin{array}{ccccc}
\sum_{j=1}^{d}J_{\varepsilon}v_{j}\xi_{j} &0  & 0 & \cdots & 0\\
0 & \sum_{j=1}^{d}J_{\varepsilon}v_{j}\xi_{j} & 0 & \cdots & 0 \\
0 & 0 & \sum_{j=1}^{d}J_{\varepsilon}v_{j}\xi_{j} & \cdots & 0 \\
\cdots &  \cdots & \cdots &  \sum_{j=1}^{d}J_{\varepsilon}v_{j}\xi_{j} & \cdots \\
0 & 0 & 0 & \cdots & \sum_{j=1}^{d}J_{\varepsilon}v_{j}\xi_{j} \\
\end{array} \right).
\end{equation}
 
Now, let $\Lambda=(1-\Delta)^{\frac{1}{2}}$, where $\Delta$ is the Laplace operator. We are ready to estimate $||\tilde{\textbf{u}}^{\varepsilon}||_{s}$. Then, 

$$\frac{1}{2}\frac{d}{dt}||\tilde{\textbf{u}}^{\varepsilon}||_{s}^{2} = (\Lambda^{s} \partial_{t} \tilde{\textbf{u}}^{\varepsilon}, \Lambda^{s} \tilde{\textbf{u}}^{\varepsilon})_{0}=-Re(\Lambda^{s} J_{\varepsilon}\textbf{P}T_{iA}J_{\varepsilon} \tilde{\textbf{u}}^{\varepsilon}, \Lambda^{s} \tilde{\textbf{u}}^{\varepsilon})_{0}$$

\begin{equation}
=-Re(\Lambda^{s}\textbf{P}T_{iA}J_{\varepsilon} \tilde{\textbf{u}}^{\varepsilon}, \Lambda^{s} J_{\varepsilon}\tilde{\textbf{u}}^{\varepsilon})_{0},
\end{equation}

where last equality follows from the commutation of the mollifiers $J_{\varepsilon}$ with the Fourier multiplier $\Lambda^{s}$. Next, 

\begin{equation}
\label{Paradifferential_2}
(\Lambda^{s}\textbf{P}T_{iA}J_{\varepsilon} \tilde{\textbf{u}}^{\varepsilon}, \Lambda^{s} J_{\varepsilon}\tilde{\textbf{u}}^{\varepsilon})_{0}=(\textbf{P}T_{iA} \Lambda^{s}J_{\varepsilon}\tilde{\textbf{u}}^{\varepsilon}, \Lambda^{s} J_{\varepsilon}\tilde{\textbf{u}}^{\varepsilon})_{0}+([\Lambda^{s}, \textbf{P}T_{iA}] J_{\varepsilon}\tilde{\textbf{u}}^{\varepsilon}, \Lambda^{s} J_{\varepsilon}\tilde{\textbf{u}}^{\varepsilon})_{0}.
\end{equation}

By the fact that both symbols $\Lambda^{s}(\xi)$ and $\textbf{P}T_{iA}(\xi, J_{\varepsilon}(\tilde{\textbf{u}}^{\varepsilon}+\bar{\textbf{u}}))$ are diagonal, from the properties of the commutator operator (see \cite{Metivier}, \cite{Grenier}, \cite{Benzoni}, \cite{Taylor}) it follows that
$[\Lambda^{s}, \textbf{P}T_{iA}]$ is an operator of order less than or equal to $s$, therefore 

$$([\Lambda^{s}, \textbf{P}T_{iA}] J_{\varepsilon}\tilde{\textbf{u}}^{\varepsilon}, \Lambda^{s} J_{\varepsilon}\tilde{\textbf{u}}^{\varepsilon})_{0} \le ||[\Lambda^{s}, \textbf{P}T_{iA}] J_{\varepsilon}\tilde{\textbf{u}}^{\varepsilon}||_{0} || \Lambda^{s} J_{\varepsilon}\tilde{\textbf{u}}^{\varepsilon}||_{0} \le c(||\tilde{\textbf{u}}^{\varepsilon}||_{s}) ||\tilde{\textbf{u}}^{\varepsilon}||_{s}^{2}.$$

It remains to deal with the first term of the right hand side of (\ref{Paradifferential_2}), namely

\begin{equation}
Re(\textbf{P}T_{iA} \Lambda^{s}J_{\varepsilon}\tilde{\textbf{u}}^{\varepsilon}, \Lambda^{s} J_{\varepsilon}\tilde{\textbf{u}}^{\varepsilon})_{0}.
\end{equation}

Looking at the operator $\textbf{P}T_{iA}$, first we neglect the remainder terms of order less than or equal to 0 which, as we have discussed before, do not influence our estimates. Now, we notice that

$$Re \left(i{\textbf{P}}A \right)=0,$$ since the symbol (\ref{PA_symbol}) is a diagonal matrix symbol of first order. Following \cite{Metivier} and using the Cauchy-Schwartz inequality, this implies that

$$Re(\textbf{P}T_{iA} \Lambda^{s}J_{\varepsilon}\tilde{\textbf{u}}^{\varepsilon}, \Lambda^{s} J_{\varepsilon}\tilde{\textbf{u}}^{\varepsilon})_{0} \le c(||\tilde{\textbf{u}}^{\varepsilon}||_{s}, \bar{\rho})||\tilde{\textbf{u}}^{\varepsilon}||_{s}^{2}.$$

Putting it all together and returning to (\ref{Paradifferential_2}), we obtain the desired estimate

\begin{equation}
\label{last_estimate}
\frac{d}{dt}||\tilde{\textbf{u}}^{\varepsilon}||_{s}^{2} \le c(||\tilde{\textbf{u}}^{\varepsilon}||_{s}, \bar{\rho}) ||\tilde{\textbf{u}}^{\varepsilon}||_{s}^{2}.
\end{equation}

Let $T_{\varepsilon}$ be the maximum time of existence of the solution to system (\ref{P_approx_compact}). We want to show that there exists a time $T>0$, which is independent of $\varepsilon$, such that $T \le T_{\varepsilon}$ for every $\varepsilon>0$. From the statement of Theorem \ref{Local Existence of Approximating Solutions_P}, there exists a constant $M$ such that $\displaystyle||\tilde{\textbf{u}}^{\varepsilon}_{0}||_{s} \le M$. Fixed a constant value $\tilde{M}>M,$ let $T_{0}^{\varepsilon} \le T_{\varepsilon}$ be a positive time such that the smooth solution $\tilde{\textbf{u}}^{\varepsilon}$ verifies

\begin{equation}
\label{infinity_bound1}
sup_{~ 0 \le \tau \le T_{0}^{\varepsilon} ~} ||\tilde{\textbf{u}}^{\varepsilon}(\tau)||_{s} \le \tilde{M}.
\end{equation}

By (\ref{last_estimate}), we get
\begin{equation}
\label{prebound1}
||\tilde{\textbf{u}}^{\varepsilon}(t)||_{s} \le ||\tilde{\textbf{u}}^{\varepsilon}_{0}||_{s} e^{c(\tilde{M}, \bar{\rho})t}
\end{equation}

for $t \in [0, T_{0}^{\varepsilon}]$. Let $T,$ with $0 < T \le T_{0}^{\varepsilon},$ be such that

\begin{equation}
M e^{c(\tilde{M}, \bar{\rho})T} \le \tilde{M},
\end{equation}

namely

\begin{equation}
\label{time_bound1}
T \le \frac{log(\frac{\tilde{M}}{M})}{c(\tilde{M}, \bar{\rho})}.
\end{equation}

Since $M, \tilde{M}, \bar{\rho}$ are independent of the parameter $\varepsilon$, estimate (\ref{time_bound1}) implies that  $T$ is independent of $\varepsilon$ and the $\varepsilon$-dependent sequence $(\tilde{\textbf{u}}^{\varepsilon})_{\varepsilon \ge 0}$ is uniformly bounded provided that \eqref{time_bound1} holds.
\end{proof}

We also need a uniform bound for the time derivatives $(\partial_{t}\tilde{\textbf{u}}^{\varepsilon}(t))_{\varepsilon \ge 0}$, at least in the low norm $L^{2}$. This is immediately given by the uniform estimate 

\begin{equation}
\label{uniform_estimate}
||\tilde{\textbf{u}}^{\varepsilon}(t)||_{s} \le c(M, \tilde{M}) ~~~~~ \text{for} ~~ t \in [0,T]
\end{equation}

just proved, which, using equation (\ref{P_approx_compact}), implies

\begin{equation}
\label{time_estimate}
||\partial_{t}\tilde{\textbf{u}}^{\varepsilon}(t)||_{s} \le c(M, \tilde{M}) ||\tilde{\textbf{u}}^{\varepsilon}(t)||_{s} \le \tilde{c}(M, \tilde{M}) ~~~~~ \text{for} ~~ t \in [0,T].
\end{equation}

\subsection{Convergence to the solution to the Compressible-Incompressible System}
This subsection is devoted to the proof of the following theorem:

\begin{theorem}
\label{Convergence}
Let $\tilde{\textbf{u}}_{0}=(\tilde{\rho}_{0}, \tilde{v}_{0})^{T}$ be the translated initial data in (\ref{translated_initial_data}), $\tilde{\textbf{u}}_{0} \in H^{m}(\mathbb{R}^{d})$ with $m > [d/2]+1$ integer. There is a positive time $T$, such that there exists the unique solution $\tilde{\textbf{u}} \in C([0,T], H^{m}(\mathbb{R}^{d})) \cap C^{1}([0,T], H^{m-1}(\mathbb{R}^{d}))$ to the compressible-incompressible system (\ref{Goal_System_Translated}). The solution $\tilde{\textbf{u}}$ to equations (\ref{Goal_System_Translated}) is the limit of a subsequence of solutions to the approximating system (\ref{P_approx_compact}) with initial data (\ref{translated_initial_data}). The incompressible pressure term $P$ satisfies (\ref{Goal_System_Translated}), namely

\begin{equation}
\label{incompressible_pressure}
\partial_{t}{\textbf{u}}+\sum_{j=1}^{d} A_{j}(\textbf{u}+\bar{\textbf{u}})\partial_{x_{j}}\textbf{u} + (0, \nabla P)^{T} = 0.
\end{equation}
\end{theorem}

\begin{proof}
In order to prove the convergence of a subsequence $(\tilde{\textbf{u}}^{\varepsilon})_{\varepsilon \ge 0}$ of solutions to system (\ref{P_approx_compact}) to solutions to (\ref{Goal_System_Translated}), let us consider the following uniform bounds that we have just proved in (\ref{uniform_estimate}) and (\ref{time_estimate}):
\begin{equation}
\label{1bound1}
sup_{~ 0 \le t \le T}||\tilde{\textbf{u}}^{\varepsilon}||_{m} \le M_{1}, 
\end{equation}
and
\begin{equation}
\label{2bound1}
sup_{~ 0 \le t \le T}||\partial_{t}\tilde{\textbf{u}}^{\varepsilon}||_{0} \le M_{2}.
\end{equation}
Since (\ref{1bound1}) and (\ref{2bound1}), the Lions-Aubin lemma (see \cite{Temam}) implies that there exists a subsequence - still denoted with $\tilde{\textbf{u}}^{\varepsilon}$- and a limit function $\tilde{\textbf{u}}^{\star}=(\tilde{\rho}^{\star}, \tilde{v}^{\star}),$ such that

\begin{equation}
\label{uniform_convergence1}
\tilde{\textbf{u}}^{\varepsilon} \rightarrow \tilde{\textbf{u}}^{\star}
\end{equation}

as $\varepsilon \rightarrow 0$ in $C([0,T], H^{m'}_{loc}(\mathbb{R}^{d}))$ with $m'<m$ and $\nabla \cdot \tilde{v}^{\star}=0$. Now, we are going to show that $\tilde{\textbf{u}}^{\star} \in L^{\infty}([0,T], H^{m}(\mathbb{R}^{d}))$. Since (\ref{1bound1}), the sequence $(\tilde{\textbf{u}}^{\varepsilon})_{\varepsilon \ge 0}$ is bounded in the reflexive Hilbert space $L^{2}([0,T], H^{m}(\mathbb{R}^{d}))$, therefore the Banach-Alaoglu theorem implies that there exists a subsequence $\tilde{\textbf{u}}^{\varepsilon'}$ and a function $\tilde{\textbf{u}}^{1}$ such that the subsequence weakly converges to $\tilde{\textbf{u}}^{1}$ in $L^{2}([0,T], H^{m}(\mathbb{R}^{d}))$. By using the uniform convergence of $\tilde{\textbf{u}}^{\varepsilon}$ to $\tilde{\textbf{u}}^{\star}$ in $C([0,T], H^{m'}_{loc}(\mathbb{R}^{d}))$, it can be easily seen that $\tilde{\textbf{u}}^{1} = \tilde{\textbf{u}}^{\star}$, i.e. $\tilde{\textbf{u}}^{\star} \in L^{2}([0,T], H^{m}(\mathbb{R}^{d}))$. 

%Now, for $h \in L^{2}(\mathbb{R}^{d})$ and $g \in H^{m}(\mathbb{R}^{d})$, consider $F_{h}g=\int_{\mathbb{R}^{d}} g {h} ~ dx$, which defines a continuous linear map on $H^{m}(\mathbb{R}^{d})$, hence, by the Riesz Representation Theorem, we have that $\forall h \in L^{2}(\mathbb{R}^{d})$ $\exists \phi(h) \in H^{m}(\mathbb{R}^{d})$ so that $F_{h}g=\int_{\mathbb{R}^{d}} g {h} ~ dx = (g, \phi(h))_{m}$, where $(\cdot, \cdot)_{m}$ denotes the inner product in $H^{m}(\mathbb{R}^{d})$. Then, we get:

%\begin{equation*}
%\int_{0}^{T} {h}v^{*} ~ dt  = \int_{0}^{T} F_{h}v^{*} ~ dt=\int_{0}^{T} (\phi(h), v^{*})_{m} ~ dt = \int_{0}^{T} lim_{\varepsilon' \rightarrow 0} (\phi(h), v^{\varepsilon'})_{m} ~ dx 
%\end{equation*} 

%\begin{equation}
%\label{convergenceL2}
%= \int_{0}^{T} ~ dt ~ lim_{\varepsilon' \rightarrow 0} \int_{\mathbb{R}^{d}} {h}v^{\varepsilon'} ~ dx = \int_{0}^{T} ~ dt \int_{\mathbb{R}^{d}} {h} v^{\infty} ~ dx,
%\end{equation} 

%where last equality of first line follows from the weak convergence just discussed and last equality of second line follows from (\ref{convergenceCHm}). Therefore $v^{*}=v^{\infty}$. 

Fix $t \in [0,T]$. Since the sequence $\tilde{\textbf{u}}^{\varepsilon}(\cdot, t)$ is uniformly bounded in $H^{m}$, it follows that there exists a subsequence $\tilde{\textbf{u}}^{\varepsilon''}(\cdot, t)$ and a limit distribution $\tilde{\textbf{u}}^{2}(\cdot, t)$ which is the weak limit of this subsequence in $H^{m}(\mathbb{R}^{d})$. Once again, the uniform convergence (\ref{uniform_convergence1}) of $\tilde{\textbf{u}}^{\varepsilon}$ implies that $\tilde{\textbf{u}}^{2}(\cdot, t)=\tilde{\textbf{u}}^{\star}(\cdot, t)$. Thus, for each $t \in [0,T]$, the limit function's norm $||\tilde{\textbf{u}}^{\star}(\cdot, t)||_{m}$ is bounded and then $\tilde{\textbf{u}}^{\star} \in L^{\infty}([0,T], H^{m}(\mathbb{R}^{d}))$.

\par\bigskip

Now, by using (\ref{1bound1}) and since $\tilde{\textbf{u}}^{\star} \in C([0,T], H^{m'}_{loc}(\mathbb{R}^{d}))$, we show that $\tilde{\textbf{u}}^{\star} \in C([0,T], H^{m'}(\mathbb{R}^{d}))$. In other words, we want to prove that, for all $\varepsilon > 0,$ there exists $\delta > 0$ such that, for any $h \le \delta$,

\begin{equation*}
||\tilde{\textbf{u}}^{\star}(\cdot, t+h)-\tilde{\textbf{u}}^{\star}(\cdot, t)||_{m'} \le \varepsilon.
\end{equation*}

To prove that, since $\tilde{\textbf{u}}^{\star} \in L^{\infty}([0,T],H^{m}(\mathbb{R}^{d}))$, it follows that for all $ \varepsilon > 0$ there exists $R=R(T)>0$ such that

\begin{equation*}
sup_{~ 0 \le t \le T ~} ||\tilde{\textbf{u}}^{\star}(\cdot, t)||_{H^{m}(\mathbb{R}^{d} / B_{R}(x_{0}))} < \frac{\varepsilon}{4},
\end{equation*}

where $B_{R}(x_{0})$ is the ball of center $x_{0} \in \mathbb{R}^{d}$ and radius $R>0$. Therefore, we have

\begin{equation*}
||\tilde{\textbf{u}}^{\star}(\cdot, t+h) - \tilde{\textbf{u}}^{\star}(\cdot, t)||_{H^{m}(\mathbb{R}^{d} / B_{R}(x_{0}))} \le \frac{\varepsilon}{2}.
\end{equation*}

Hence, we get

\begin{equation*}
||\tilde{\textbf{u}}^{\star}(\cdot, t+h)-\tilde{\textbf{u}}^{\star}(\cdot, t)||_{m'} 
\end{equation*}

\begin{equation*}
\le ||\tilde{\textbf{u}}^{\star}(\cdot, t+h)-\tilde{\textbf{u}}^{\star}(\cdot, t)||_{H^{m'}(B_{R}(x_{0}))} + ||\tilde{\textbf{u}}^{\star}(\cdot, t+h) - \tilde{\textbf{u}}^{\star}(\cdot, t)||_{H^{m'}(\mathbb{R}^{d} / B_{R}(x_{0}))} < \frac{\varepsilon}{2} + \frac{\varepsilon}{2}=\varepsilon,
\end{equation*}

where inequality 
\begin{equation*}
||\tilde{\textbf{u}}^{\star}(\cdot, t+h)-\tilde{\textbf{u}}^{\star}(\cdot, t)||_{H^{m'}(B_{R}(x_{0})))} \le \frac{\varepsilon}{2}
\end{equation*}

holds since $\tilde{\textbf{u}}^{\star} \in C([0,T], H^{m'}_{loc}(\mathbb{R}^{d}))$. Then, $\tilde{\textbf{u}}^{\star} \in C([0,T], H^{m'}(\mathbb{R}^{d}))$. Next, let $\psi \in C_{c}^{\infty}((0,T))$ and $\phi=(\rho, v)^{T}$ so that  $v \in \mathcal{V}^{0}=\{ v \in L^{2}(\mathbb{R}^{d}) ~|~ \nabla \cdot v = 0\}$ rapidly decreasing. Writing a weak formulation of system (\ref{P_approx_compact}), we have

\begin{equation}
\label{weak_solution1}
\int_{0}^{T} - \psi'(t) (\tilde{\textbf{u}}^{\varepsilon}, \phi)_{0}  ~ dt  + \sum_{j=1}^{d}\int_{0}^{T} \psi(t) (\textbf{P} J_{\varepsilon}A_{j}(J_{\varepsilon}(\tilde{\textbf{u}}^{\varepsilon}+\bar{\textbf{u}}))\partial_{x_{j}}J_{\varepsilon}\tilde{\textbf{u}}^{\varepsilon},  \phi)_{0}  ~ dt = 0.
\end{equation}

By using the uniform convergence of $\tilde{\textbf{u}}^{\varepsilon} \rightarrow \tilde{\textbf{u}}^{\star}$ in $C([0,T], H^{m'}_{loc}(\mathbb{R}^{d}))$ and recalling that $\tilde{\textbf{u}}^{\star} \in L^{\infty}([0,T], H^{m}(\mathbb{R}^{d}))$, we have   

\begin{equation}
(\tilde{\textbf{u}}^{\varepsilon}, \phi) \rightarrow (\tilde{\textbf{u}}^{\star}, \phi)_{0}.
\end{equation}

Now, for $j=1, \cdots, d$, it holds 
$$|(\textbf{P}J_{\varepsilon}A_{j}(J_{\varepsilon}(\tilde{\textbf{u}}^{\varepsilon}+\bar{\textbf{u}}))\partial_{x_{j}}J_{\varepsilon}\tilde{\textbf{u}}^{\varepsilon}-\textbf{P}A_{j}(\tilde{\textbf{u}}^{\star}+\bar{\textbf{u}})\partial_{x_{j}}\tilde{\textbf{u}}^{\star}, \phi)_{0}|$$

$$\le |(\textbf{P}J_{\varepsilon}A_{j}(J_{\varepsilon}(\tilde{\textbf{u}}^{\varepsilon}+\bar{\textbf{u}}))\partial_{x_{j}}(J_{\varepsilon}\tilde{\textbf{u}}^{\varepsilon}-\tilde{\textbf{u}}^{\star}), \phi)_{0}|$$

$$+|((\textbf{P}J_{\varepsilon}A_{j}(J_{\varepsilon}(\tilde{\textbf{u}}^{\varepsilon}+\bar{\textbf{u}}))-\textbf{P}A_{j}(\tilde{\textbf{u}}^{\star}+\bar{\textbf{u}}))\partial_{x_{j}}\tilde{\textbf{u}}^{\star}, \phi)_{0}|$$

\begin{equation}
\le M_{1} |(\partial_{x_{j}}(J_{\varepsilon}\tilde{\textbf{u}}^{\varepsilon}-\tilde{\textbf{u}}^{\star}), \phi)_{0}| + |(\textbf{P}(J_{\varepsilon} A_{j}(J_{\varepsilon}(\tilde{\textbf{u}}^{\varepsilon}+\bar{\textbf{u}}))-A_{j}(\tilde{\textbf{u}}^{\varepsilon}+\bar{\textbf{u}}))\partial_{x_{j}}\tilde{\textbf{u}}^{\star}, \phi)_{0}|.
\end{equation}

Therefore, again, since
$\tilde{\textbf{u}}^{\varepsilon} \rightarrow \tilde{\textbf{u}}^{\star}$ in $C([0,T], H^{m'}_{loc}(\mathbb{R}^{d}))$ and using that $\tilde{\textbf{u}}^{\star} \in L^{\infty}([0,T], H^{m}(\mathbb{R}^{d}))$, it follows the convergence

\begin{equation}
\label{L2convergence}
\int_{0}^{T} \psi(t) (\textbf{P} J_{\varepsilon}A_{j}(J_{\varepsilon}(\tilde{\textbf{u}}^{\varepsilon}+\bar{\textbf{u}}))\partial_{x_{j}}J_{\varepsilon}\tilde{\textbf{u}}^{\varepsilon},  \phi)_{0} \rightarrow\int_{0}^{T} \psi(t) (\textbf{P} A_{j}(\tilde{\textbf{u}}^{\star}+\bar{\textbf{u}})\partial_{x_{j}}\tilde{\textbf{u}}^{\star},  \phi)_{0}
\end{equation}

uniformly on $[0,T]$. This way, passing to the limit in (\ref{weak_solution1}), we obtain 

\begin{equation}
\label{weak_solution_limit1}
\int_{0}^{T} - \psi'(t) (\tilde{\textbf{u}}^{\star}, \phi)_{0}  ~ dt  + \sum_{j=1}^{d}\int_{0}^{T} \psi(t) (\textbf{P} A_{j}(\tilde{\textbf{u}}^{\star}+\bar{\textbf{u}})\partial_{x_{j}}\tilde{\textbf{u}}^{\star},  \phi)_{0}  ~ dt = 0.
\end{equation}

Convergence of (\ref{weak_solution1}) to (\ref{weak_solution_limit1}) and (\ref{2bound1}) imply, through a standard density argument and passing to a subsequence, the weak$^{*}$ convergence $\partial_{t}\tilde{\textbf{u}}^{\varepsilon} \rightharpoonup^{*}\partial_{t}\tilde{\textbf{u}}^{\star}$ in $L^{\infty}([0,T], L^{2}(\mathbb{R}^{d}))$. Now, the weak formulation (\ref{weak_solution_limit1}), in accordance with (\ref{Proj_Formulation_no_Symm}), yields

\begin{equation}
\label{distribution_solution_limit}
\partial_{t}\tilde{\textbf{u}}^{\star}+\sum_{j=1}^{d}\textbf{P}(A_{j}(\tilde{\textbf{u}}^{\star}+\bar{\textbf{u}})\partial_{x_{j}}\tilde{\textbf{u}}^{\star})=0
\end{equation}

in the sense of distributions. Thus, recalling that $\tilde{\textbf{u}}^{\star} \in L^{\infty}([0,T], H^{m}(\mathbb{R}^{d}))$ and using (\ref{distribution_solution_limit}), we obtain also that $\tilde{\textbf{u}}^{\star} \in Lip([0,T], H^{m-1}(\mathbb{R}^{d}))$. Moreover, from (\ref{distribution_solution_limit}) and the Hodge decomposition theorem, it follows that there exists $\nabla P^{\star} \in L^{\infty}([0,T], H^{m-1}(\mathbb{R}^{d}))$ such that

\begin{equation}
\label{distribution_solution_limit2}
\partial_{t}\tilde{\textbf{u}}^{\star} + \sum_{j=1}^{d}A_{j}(\tilde{\textbf{u}}^{\star}+\bar{\textbf{u}})\partial_{x_{j}}\tilde{\textbf{u}}^{\star}=(0, -\nabla P^{\star})^{T}.
\end{equation}

Equation (\ref{distribution_solution_limit2}) means that $\tilde{\textbf{u}}^{\star} \in L^{\infty}([0,T], H^{m}(\mathbb{R}^{d})) \cap  Lip([0,T], H^{m-1}(\mathbb{R}^{d}))$ is a weak solution to the compressible-incompressible system (\ref{Goal_System}).

Now, since $\tilde{\textbf{u}}^{\star} \in C([0,T], H^{m'}(\mathbb{R}^{d}))$, it follows that  $\tilde{\textbf{u}}^{\star} \in C_{w}([0,T], H^{m'}(\mathbb{R}^{d})),$ i.e. for all $ \varepsilon > 0,$ for all $\phi' \in H^{-m'}(\mathbb{R}^{d})$ there exists $\delta > 0$ such that 

\begin{equation*}
|(\tilde{\textbf{u}}^{\star}(\cdot, t+h)-\tilde{\textbf{u}}^{\star}(\cdot, t), \phi')_{-m', m'}| \le \frac{\varepsilon}{2}.
\end{equation*}

Furthermore, by the density of $H^{-m'} \subset H^{-m}$ ($m'<m$), it follows that for all $\varepsilon > 0$ and for all $\phi \in H^{-m}(\mathbb{R}^{d})$, there exists $\phi' \in H^{-m'}(\mathbb{R}^{d})$ such that 

\begin{equation*}
||\phi - \phi'||_{-m} \le \frac{\varepsilon}{4M_{1}}, 
\end{equation*}

where $M_{1}$ is the uniform bound in (\ref{1bound1}). Putting it all together, we get

\begin{equation*}
|(\tilde{\textbf{u}}^{\star}(\cdot, t+h)-\tilde{\textbf{u}}^{\star}(\cdot, t), \phi)_{-m, m}| 
\end{equation*}

\begin{equation*}
\le |(\tilde{\textbf{u}}^{\star}(\cdot, t+h)-\tilde{\textbf{u}}^{\star}(\cdot, t), \phi-\phi')_{-m, m}| + |(\tilde{\textbf{u}}^{\star}(\cdot, t+h)-\tilde{\textbf{u}}^{\star}(\cdot, t), \phi')_{-m', m'}|
\end{equation*}

\begin{equation*}
\le 2M_{1} ||\phi -\phi'||_{-m} +  |(\tilde{\textbf{u}}^{\star}(\cdot, t+h)-\tilde{\textbf{u}}^{\star}(\cdot, t), \phi')_{-m', m'}| \le \varepsilon,
\end{equation*}

which implies $\tilde{\textbf{u}}^{\star} \in C_{w}([0,T], H^{m}(\mathbb{R}^{d})).$ Finally, we have obtained that $\tilde{\textbf{u}}^{\star} \in L^{\infty}([0,T], H^{m}(\mathbb{R}^{d})) \cap Lip([0,T], H^{m-1}(\mathbb{R}^{d})) \cap C_{w}([0,T], H^{m}(\mathbb{R}^{d}))$.

\bigskip

Passing to a subsequence and recalling 

\begin{equation}
\partial_{t}\tilde{\textbf{u}}^{\star} \rightharpoonup^{*} \partial_{t}\tilde{\textbf{u}}^{\star} ~~ in ~~ L^{\infty}([0,T],L^{2}(\mathbb{R}^{d})),
\end{equation}

and (\ref{weak_solution1})-(\ref{weak_solution_limit1}), we get

\begin{equation}
\label{convergence}
\partial_{t}\tilde{\textbf{u}}^{\varepsilon} + \sum_{j=1}^{d}\textbf{P} J_{\varepsilon}A_{j}(\tilde{\textbf{u}}^{\varepsilon}+\bar{\textbf{u}})\partial_{x_{j}}J_{\varepsilon}\tilde{\textbf{u}}^{\varepsilon} \rightharpoonup^{*} \partial_{t}\tilde{\textbf{u}}^{\star} + \sum_{j=1}^{d}\textbf{P} A_{j}(\tilde{\textbf{u}}^{\star}+\bar{\textbf{u}})\partial_{x_{j}}\tilde{\textbf{u}}^{\star}.
\end{equation}

in  $\displaystyle L^{\infty}([0,T],L^{2}(\mathbb{R}^{d}))$, where $ \partial_{t}\tilde{\textbf{u}}^{\star}  \in C_{w}([0,T], H^{m-1}(\mathbb{R}^{d}))$ since (\ref{distribution_solution_limit}) holds and $ \tilde{\textbf{u}}^{\star}  \in C_{w}([0,T], H^{m}(\mathbb{R}^{d}))$. Recalling (\ref{1bound1}), we have

\begin{equation}
sup_{~0\le t \le T~} ||\nabla P^{\varepsilon}||_{0} \le c(M_{1}, M_{2}).
\end{equation}

Therefore, by using the Banach-Alaoglu theorem in $L^{\infty}([0,T], L^{2}(\mathbb{R}^{d}))$ and recalling (\ref{distribution_solution_limit2}), we get

\begin{equation}
\label{PressureConvergence}
{\nabla P^{\varepsilon}} \rightharpoonup^{*} \nabla P^{\star} ~~ in ~~ L^{\infty}([0,T], L^{2}(\mathbb{R}^{d})),
\end{equation}

with $P^{*} \in C_{w}([0,T], H^{m-1}(\mathbb{R}^{d}))$ since (\ref{distribution_solution_limit2}) and the regularity of $\partial_{t}\tilde{\textbf{u}}^{*}$ and $\tilde{\textbf{u}}^{\star}$.

The additional regularity $C([0,T], H^{m}(\mathbb{R}^{d})) \cap C^{1}([0,T], H^{m-1}(\mathbb{R}^{d}))$ can be achieved in a standard way, following \cite{Bertozzi}, from $\tilde{\textbf{u}}^{\star} \in  C_{w}([0,T], H^{m}(\mathbb{R}^{d}))$. We sketch the proof. First, it is sufficient to prove that $\tilde{\textbf{u}}^{\star} \in C([0,T], H^{m}(\mathbb{R}^{d}))$ since the regularity $C^{1}([0,T], H^{m-1}(\mathbb{R}^{d}))$ follows directly from equations (\ref{distribution_solution_limit2}). Moreover, we only need to prove  the continuity of $\tilde{\textbf{u}}^{\star}$ in the strong norm $||\cdot||_{m}$ at time $t=0$, in fact the same argument can be adapted to any other time $\tilde{T}$, $0 \le \tilde{T} \le T$. Furthermore, since system (\ref{distribution_solution_limit2}) is time reversible, we only need to prove the right continuity at time $t=0$ in the strong norm $||\cdot||_{m}$. Recalling (\ref{1bound1}) and passing to a subsequence, we have that

\begin{equation}
limsup_{\varepsilon \rightarrow 0} ||\tilde{\textbf{u}}^{\varepsilon}||_{m} \ge ||\tilde{\textbf{u}}^{\star}||_{m}.
\end{equation}

Moreover, from (\ref{prebound1}), it holds that

\begin{equation}
||\tilde{\textbf{u}}^{\varepsilon}||_{m} \le e^{c(\tilde{M}, \bar{\rho})t} ||\tilde{\textbf{u}}^{\varepsilon}_{0}||_{m}.
\end{equation}

This implies 

\begin{equation}
sup_{~ 0 \le t \le T ~} ||\tilde{\textbf{u}}^{\varepsilon}||_{m} - ||\tilde{\textbf{u}}^{\varepsilon}_{0}||_{m} \le e^{c(\tilde{M}, \bar{\rho})T} ||\tilde{\textbf{u}}^{\varepsilon}_{0}||_{m} - ||\tilde{\textbf{u}}^{\varepsilon}_{0}||_{m}.
\end{equation}

Because of that, from Theorem \ref{Local Existence of Approximating Solutions_P}, system (\ref{P_approx_compact}) is associated to the original initial data in (\ref{translated_initial_data}), for every $\varepsilon$ we have:

$$||\tilde{\textbf{u}}^{\varepsilon}_{0}||_{m} = ||\tilde{\textbf{u}}_{0}||_{m}.$$

Last estimates give

\begin{equation}
limsup_{~t \rightarrow 0^{+}~}  ||\tilde{\textbf{u}}^{\star}||_{m} \le ||\tilde{\textbf{u}}_{0}||_{m}.
\end{equation}

Now, since $\tilde{\textbf{u}}^{\star} \in C_{w}([0,T], H^{m}(\mathbb{R}^{d}))$, it follows that

\begin{equation}
liminf_{~t \rightarrow 0^{+}~} ||\tilde{\textbf{u}}^{\star}||_{m} \ge ||\tilde{\textbf{u}}_{0}||_{m}.
\end{equation}

In particular, 

\begin{equation}
lim_{~t \rightarrow 0^{+}~}||\tilde{\textbf{u}}^{\star}(t)||_{m} = ||\tilde{\textbf{u}}_{0}||_{m}.
\end{equation}

Then, the strong right continuity at $t=0$ is proved. 
%\begin{remark}
%It is not difficult to estimate $\partial_{t}\tilde{\textbf{u}}^{\star}$ in $L^{\infty}([0,T], H^{m-2}(\mathbb{R}^{d}))$ and prove (\ref{convergence}), (\ref{PressureConvergence}) in $L^{\infty}([0,T], H^{m-2}(\mathbb{R}^{d})).$ 
%\end{remark}
\end{proof}

\begin{remark}
\label{remark1}
We point out that this kind of approximation does not work on system (\ref{density_dependent}), because of the fact that, we cannot eliminate the term $\frac{\nabla P}{\rho}$, which is not a gradient, by applying the projector operator.
\end{remark}

\section{The continuous projection approximation}
First of all, we want to point out that, although we use again the Leray projector, the idea inside this other kind of approximation is quite different from that discussed before. Roughly speaking, the main feature is that we will apply the projector operator in a completely different way, somehow treating it like a source term. Obviously, the first order symbol of this new approximation has to satisfies the hyperbolicity property. This is the reason why we go back to the original formulation (\ref{Goal_Compact}) of our problem and its matrices $A_{j}$ in (\ref{A11}) and (\ref{Ajjj}). Recalling the general setting of system (\ref{Goal_Compact}) presented in Section 2, notice that, in order to have real and semisimple eigenvalues, we have to make some assumptions on the function $f(\textbf{u})$. We are led to give a definition of \emph{admissible} scalar functions $f(\textbf{u})$.

\begin{definition}
\label{f}
The scalar function $f(\textbf{u})$ in (\ref{Goal_Compact}) is an \emph{admissible} function if 
\begin{itemize}
\item $f(\textbf{u})$ is strictly positive;
\item $\nabla_{v} f(\textbf{u}) = \alpha (\rho, |v|) v$, where $\alpha$ is a positive and continuous scalar function, only depending on the density $\rho$ and the norm $|v|$ of the velocity field.
\end{itemize}
An example of an admissible function is given by
\begin{equation}
f(\textbf{u})=\bar{f}+\beta(\rho, v^{2}), ~~~~ \beta(\rho_{0}, v_{0}^{2}) \ge 0,
\end{equation}
where $\bar{f}$ is a constant value and $\nabla_{v}\beta(\rho, v^{2})=2\partial_{v^{2}}\beta(\rho, v^{2}) v$, with $\partial_{v^{2}}\beta(\rho, v^{2})\ge 0$.
Last condition of this definition is essential to have uniformly bounded energy estimates for the approximation that we are going to define, as we will see.
\end{definition}
Now, we have the right framework for our problem and so, following \cite{Bertozzi} and \cite{Temam}, we look for a suitable approximation  to the compressible-incompressible system (\ref{Goal_System}), which is:

\begin{equation}
\label{Mollifiers_system1}
A_{0}(J_{\varepsilon}\textbf{u}^{\varepsilon})\partial_{t}\textbf{u}^{\varepsilon}+\sum_{j=1}^{d}J_{\varepsilon}A_{0}A_{j}(J_{\varepsilon}\textbf{u}^{\varepsilon})\partial_{x_{j}}J_{\varepsilon}\textbf{u}^{\varepsilon}+\left( \begin{array}{c}
0 \\
 \nabla P^{\varepsilon} \\
\end{array} \right)=0,
\end{equation}

where $\textbf{u}^{\varepsilon}:=(\rho^{\varepsilon}, v^{\varepsilon})$ and $v^{\varepsilon}$ is no more divergence free. Now, we choose the approximating sequence $\nabla P^{\varepsilon}$ so that, for each fixed $\varepsilon$, $\nabla P^{\varepsilon}$ is proportional to the gradient part of $v^{\varepsilon}$. Namely, using the Hodge decomposition theorem, we can set

\begin{equation}
\label{Hodge}
v^{\varepsilon}=\mathbb{P}v^{\varepsilon}+\varepsilon \nabla P^{\varepsilon}.
\end{equation}

This way, $$\nabla P^{\varepsilon}=\frac{(\mathbb{I}-\mathbb{P})v^{\varepsilon}}{\varepsilon}.$$ Then, the approximate equation becomes 

\begin{equation}
\label{Mollifiers_system}
A_{0}(J_{\varepsilon}\textbf{u}^{\varepsilon})\partial_{t}\textbf{u}^{\varepsilon}+\sum_{j=1}^{d}J_{\varepsilon}A_{0}A_{j}(J_{\varepsilon}\textbf{u}^{\varepsilon})\partial_{x_{j}}J_{\varepsilon}\textbf{u}^{\varepsilon}= - \left( \begin{array}{c}
0 \\
 \frac{(\mathbb{I}-\mathbb{P})v^{\varepsilon}}{\varepsilon} \\
\end{array} \right),
\end{equation}

with initial data 

\begin{equation}
\label{approximating_initial}
\rho_{0}^{\varepsilon}(x)=\rho_{0}(0,x), ~~~~  v_{0}^{\varepsilon}(0, x)=v_{0}(x)+\varepsilon v_{0}^{1}(x),
\end{equation}

where $\rho_{0}, v_{0}$ are the initial data in (\ref{initial}) - $v_{0} \in \mathcal{V}=\{ v \in H^{m}(\mathbb{R}^{d}) : \nabla \cdot v=0\}$ - and $v_{0}^{1}(x)$ is  chosen in $H^{m}(\mathbb{R}^{d})$. 

\begin{remark}
Similarly to the incompressible limit of the incompressible Euler equations in \cite{Majda}, \cite{Klainerman}, the "slightly compressible" form of the initial data in (\ref{approximating_initial}) guarantees the uniform - in $\varepsilon$- bound of the time derivative of $v^{\varepsilon}$ in the $L^{2}$-norm, as we will see later.
\end{remark}

\begin{remark}
\label{Transalted_system}
Taking into account the translation we have made in system (\ref{Goal_System_Translated}), we are going to translate also our approximating system (\ref{Mollifiers_system}) and the related initial data (\ref{approximating_initial}). Setting $\tilde{\textbf{u}}^{\varepsilon}=(\tilde{\rho}^{\varepsilon}, \tilde{v}^{\varepsilon})=(\rho^{\varepsilon}-\bar{\rho}, v^{\varepsilon}),$ the formulation of the approximation that we want to consider is the following:

\begin{equation}
\label{Goal_System_Translated_Approx1}
A_{0}(J_{\varepsilon}(\tilde{\textbf{u}}^{\varepsilon}+\bar{\textbf{u}})) \partial_{t}\tilde{\textbf{u}}^{\varepsilon}+\sum_{j=1}^{d}J_{\varepsilon} A_{0}A_{j}(J_{\varepsilon}(\tilde{\textbf{u}}^{\varepsilon}+\bar{\textbf{u}}))\partial_{x_{j}}J_{\varepsilon}\tilde{\textbf{u}}^{\varepsilon}= - \left( \begin{array}{c}
0 \\
\frac{(\mathbb{I}-\mathbb{P})v^{\varepsilon}}{\varepsilon} \\
\end{array} \right),
\end{equation}

with translated initial data

\begin{equation}
\label{Compact_translated_initial_data}
\tilde{\textbf{u}}^{\varepsilon}_{0}=(\tilde{\rho}^{\varepsilon}_{0}, \tilde{{v}}^{\varepsilon}_{0})^{T},
\end{equation}

where

\begin{equation}
\label{translated_initial_data_original}
\tilde{\rho}^{\varepsilon}_{0}(x)=\rho^{\varepsilon}_{0}(x)-\bar{\rho}, ~~~~~ \tilde{v}^{\varepsilon}_{0}(x)=v_{0}^{\varepsilon}(x),
\end{equation}
and $\rho^{\varepsilon}_{0}$ and $v^{\varepsilon}_{0}$ the approximating initial data in (\ref{approximating_initial}).
\end{remark}

Again, we use the Picard theorem on Banach spaces to get local solutions to the approximation (\ref{Goal_System_Translated_Approx1}) with initial data (\ref{Compact_translated_initial_data}). System (\ref{Goal_System_Translated_Approx1}) reduces to an ODE:

\begin{equation}
\label{ODE}
\partial_{t}\tilde{\textbf{u}}^{\varepsilon} = F^{\varepsilon}(\tilde{\textbf{u}}^{\varepsilon}),  ~~~~~~~~~~\tilde{\textbf{u}}^{\varepsilon}|_{t=0}=\tilde{\textbf{u}}^{\varepsilon}_{0}(x),
\end{equation}

where 

\begin{equation}
\label{Function_ODE}
F^{\varepsilon}(\tilde{\textbf{u}}^{\varepsilon}) = -\sum_{j=1}^{d}A_{0}^{-1}J_{\varepsilon}A_{0}A_{j}(J_{\varepsilon}(\tilde{\textbf{u}^{\varepsilon}}+\bar{\textbf{u}}))\partial_{x_{j}}J_{\varepsilon}\tilde{\textbf{u}}^{\varepsilon} - \left( \begin{array}{c}
0 \\
 \frac{(\mathbb{I}-\mathbb{P})v^{\varepsilon}}{\varepsilon} \\
\end{array} \right) =:F_{1}^{\varepsilon}(\tilde{\textbf{u}}^{\varepsilon}) - F_{2}^{\varepsilon}(\tilde{\textbf{u}}^{\varepsilon}).
\end{equation}

We want to prove the following theorem;

\begin{theorem}(Local existence of approximating solutions for the second type of approximation)
\label{Local Existence of Approximating Solutions}
Let $\tilde{\textbf{u}}^{\varepsilon}_{0}=(\tilde{\rho}_{0}^{\varepsilon}, \tilde{v}_{0}^{\varepsilon})^{T} \in H^{m}(\mathbb{R}^{d})$ as in (\ref{Compact_translated_initial_data}) and $m \in \mathbb{N}$, with $m > [d/2]+1$. Then, for any $\varepsilon > 0$, there exists a time $T$, independent of $\varepsilon$, such that system (\ref{Goal_System_Translated_Approx1}) has a unique solution $\tilde{\textbf{u}}^{\varepsilon}=(\tilde{\rho}^{\varepsilon}, \tilde{v}^{\varepsilon})^{T} \in C^{1}([0, T], H^{m}(\mathbb{R}^{d}))$.
\end{theorem}

\begin{proof}
First, we prove existence and uniqueness using again the Picard theorem. Then, we verify that the time of local existence $ T_{\varepsilon}$ can be bounded from below by a positive time $T$ independent of $\varepsilon$. In this context, the preliminary considerations that assure the applicability of the Picard Theorem  have been already discussed in Section 2, then we omit them. Regarding the Lipschitzianity of $F_{1}^{\varepsilon}$, compared to Section 2, here the only difference is represented by the multiplication by the symmetrizer $A_{0}$ and its inverse matrix $A_{0}^{-1}$ in (\ref{Function_ODE}). Looking at $F_{2}^{\varepsilon}$ in (\ref{Function_ODE}), from the properties of the pseudodifferential operators in \cite{Metivier}, \cite{Taylor} and \cite{Benzoni}, we have

\begin{equation*}
||F_{2}^{\varepsilon}(\tilde{\textbf{u}}_{1})-F_{2}^{\varepsilon}(\tilde{\textbf{u}}_{2})||_{m} = \left( \begin{array}{c}
0 \\
 \frac{||(\mathbb{I}-\mathbb{P})(\tilde{v}_{1}-\tilde{v}_{2})||_{m}}{\varepsilon} \\
\end{array} \right)
\le \left( \begin{array}{c}
0 \\
 \frac{1}{\varepsilon}||\tilde{v}_{1}-\tilde{v}_{2}||_{m} \\
\end{array} \right)
\le \frac{1}{\varepsilon}||\tilde{\textbf{u}}_{1}-\tilde{\textbf{u}}_{2}||_{m}.
\end{equation*}

Putting it all together, we get 
\begin{equation}
\label{Diff_Picard}
||F^{\varepsilon}(\tilde{\textbf{u}}_{1})-F^{\varepsilon}(\tilde{\textbf{u}}_{2})||_{m} \le c(||\tilde{\textbf{u}}_{1}||_{m}, ||\tilde{\textbf{u}}_{2}||_{m}, \bar{\rho}, \varepsilon^{-1}) ||\tilde{\textbf{u}}_{1}-\tilde{\textbf{u}}_{2}||_{m}.
\end{equation}

Thus, for fixed $\varepsilon$, $F^{\varepsilon}$ is locally Lipschitz continuous on any open set

\begin{equation}
\mathcal{U}^{M}=\{\tilde{\textbf{u}}^{\varepsilon} \in H^{m}(\mathbb{R}^{d}) : ||\tilde{\textbf{u}}^{\varepsilon}||_{m} \le M \}.
\end{equation}

Therefore, the Picard theorem implies that there exists a unique solution $\tilde{\textbf{u}}^{ \varepsilon} \in C^{1}([0, T_{\varepsilon}), \mathcal{U}^{M})$ for any $T_{\varepsilon} > 0$.

\par\bigskip
Following the path of the proof given in Section 2, we need a uniform bound on $\tilde{\textbf{u}}^{\varepsilon}$ in the higher $H^{m}$-norm. Since (\ref{Goal_System_Translated_Approx1}), we have

\begin{equation*}
\partial_{t}\tilde{\textbf{u}}^{\varepsilon} = -\sum_{j=1}^{d}A_{0}^{-1}J_{\varepsilon}A_{0}A_{j}(J_{\varepsilon}(\tilde{\textbf{u}^{\varepsilon}}+\bar{\textbf{u}}))\partial_{x_{j}}J_{\varepsilon}\tilde{\textbf{u}}^{\varepsilon} - \left( \begin{array}{c}
0 \\
 \frac{(\mathbb{I}-\mathbb{P})v^{\varepsilon}}{\varepsilon} \\
\end{array} \right)
\end{equation*}

Taking the $\alpha$-derivative for $|\alpha| \le m$, we get 
\begin{equation}
\partial_{t} D^{\alpha} \tilde{\textbf{u}}^{\varepsilon} + \sum_{j=1}^{d}A_{0}^{-1}J_{\varepsilon}A_{0}A_{j}(J_{\varepsilon}(\tilde{\textbf{u}^{\varepsilon}}+\bar{\textbf{u}}))\partial_{x_{j}}D^{\alpha}J_{\varepsilon}\tilde{\textbf{u}}^{\varepsilon}+\left( \begin{array}{c}
0 \\
 \frac{(\mathbb{I}-\mathbb{P})D^{\alpha}v^{\varepsilon}}{\varepsilon} \\
\end{array} \right)=F_{\alpha},
\end{equation}

where

\begin{equation}
F_{\alpha}=-\sum_{j=1}^{d}[D^{\alpha}(A_{0}^{-1}J_{\varepsilon}A_{0}A_{j}(J_{\varepsilon}(\tilde{\textbf{u}^{\varepsilon}}+\bar{\textbf{u}}))\partial_{x_{j}}J_{\varepsilon}\tilde{\textbf{u}}^{\varepsilon})-A_{0}^{-1}J_{\varepsilon}A_{0}A_{j}(J_{\varepsilon}(\tilde{\textbf{u}^{\varepsilon}}+\bar{\textbf{u}}))\partial_{x_{j}}D^{\alpha}J_{\varepsilon}\tilde{\textbf{u}}^{\varepsilon}].
\end{equation}

Multiplying by $D^{\alpha}\tilde{\textbf{u}}^{ \varepsilon}$ through the $A_{0}$ inner product $(A_{0} \cdot, \cdot)_{0}$, where $A_{0}$ is the symmetrizer in (\ref{Symmetrizer1}), and using the symmetric property of mollifiers, we obtain

\begin{equation*}
\frac{1}{2}\frac{d}{dt}(A_{0}(J_{\varepsilon}(\tilde{\textbf{u}}^{\varepsilon}+\bar{\textbf{u}}))D^{\alpha}J_{\varepsilon}\tilde{\textbf{u}}^{\varepsilon}, D^{\alpha}J_{\varepsilon}\tilde{\textbf{u}}^{\varepsilon})_{0} + \frac{1}{\varepsilon}((\mathbb{I}-\mathbb{P})D^{\alpha} v^{\varepsilon}, D^{\alpha} v^{\varepsilon})_{0}
\end{equation*}

\begin{equation*}
=\frac{1}{2}(\partial_{t}A_{0}(J_{\varepsilon}(\tilde{\textbf{u}}^{\varepsilon}+\bar{\textbf{u}}))D^{\alpha}J_{\varepsilon}\tilde{\textbf{u}}^{\varepsilon}, D^{\alpha}J_{\varepsilon}\tilde{\textbf{u}}^{\varepsilon})_{0} + \sum_{j=1}^{d}(\partial_{x_{j}}(A_{0}A_{j}(J_{\varepsilon}(\tilde{\textbf{u}}^{\varepsilon}+\bar{\textbf{u}}))D^{\alpha}J_{\varepsilon}\tilde{\textbf{u}}^{\varepsilon}, D^{\alpha}J_{\varepsilon}\tilde{\textbf{u}}^{\varepsilon})_{0}
\end{equation*}

\begin{equation}
+(A_{0}(J_{\varepsilon}(\tilde{\textbf{u}}^{\varepsilon}+\bar{\textbf{u}}))F_{\alpha}, D^{\alpha}\tilde{\textbf{u}}^{\varepsilon})_{0}.
\end{equation}

This implies that

\begin{equation*}
\frac{1}{2}\frac{d}{dt}(A_{0}(J_{\varepsilon}(\tilde{\textbf{u}}^{\varepsilon}+\bar{\textbf{u}}))D^{\alpha}J_{\varepsilon}\tilde{\textbf{u}}^{\varepsilon}, D^{\alpha}J_{\varepsilon}\tilde{\textbf{u}}^{\varepsilon})_{0} + \frac{1}{\varepsilon}((\mathbb{I}-\mathbb{P})D^{\alpha} v^{\varepsilon}, D^{\alpha} v^{\varepsilon})_{0}
\end{equation*}

\begin{equation}
\le c(|\tilde{\textbf{u}}^{\varepsilon}|_{\infty}, |\nabla \tilde{\textbf{u}}^{\varepsilon}|_{\infty}, \bar{\rho})||D^{\alpha}\tilde{\textbf{u}}^{\varepsilon}||_{0}^{2} + c(|\tilde{\textbf{u}}^{\varepsilon}|_{\infty})||F_{\alpha}||_{0}||D^{\alpha}\tilde{\textbf{u}}^{\varepsilon}||_{0}.
\end{equation}

Now

\begin{equation*}
||F_{\alpha}||_{0}=\Bigg|\Bigg|\sum_{j=1}^{d}[D^{\alpha}(A_{0}^{-1}J_{\varepsilon}A_{0}A_{j}(J_{\varepsilon}(\tilde{\textbf{u}^{\varepsilon}}+\bar{\textbf{u}}))\partial_{x_{j}}J_{\varepsilon}\tilde{\textbf{u}}^{\varepsilon})-A_{0}^{-1}J_{\varepsilon}A_{0}A_{j}(J_{\varepsilon}(\tilde{\textbf{u}^{\varepsilon}}+\bar{\textbf{u}}))\partial_{x_{j}}D^{\alpha}J_{\varepsilon}\tilde{\textbf{u}}^{\varepsilon}]\Bigg|\Bigg|_{0}
\end{equation*}

\begin{equation*}
\le \sum_{j=1}^{d}\{|D(A_{0}^{-1}J_{\varepsilon}A_{0}A_{j}(\tilde{\textbf{u}^{\varepsilon}}+\bar{\textbf{u}}))|_{\infty} ||D^{m-1}\partial_{x_{j}}J_{\varepsilon}\tilde{\textbf{u}}^{\varepsilon}||_{0} 
\end{equation*}

\begin{equation}
\label{quasi_last}
+ ||D^{m}(A_{0}^{-1}J_{\varepsilon}A_{0}A_{j}(\tilde{\textbf{u}^{\varepsilon}}+\bar{\textbf{u}}))||_{0}   |\partial_{x_{j}}J_{\varepsilon}\tilde{\textbf{u}}^{\varepsilon}|_{\infty}\}
\end{equation}

and then, using Remark \ref{m_estimate} in (\ref{quasi_last}), we get the inequality

\begin{equation}
\le c(|\tilde{\textbf{u}}^{\varepsilon}|_{\infty}, |\nabla \tilde{\textbf{u}}^{\varepsilon}|_{\infty}, \bar{\rho})||D^{m}\tilde{\textbf{u}}^{\varepsilon}||_{0}^{2}.
\end{equation}

Thus, we have

\begin{equation*}
\frac{1}{2}\frac{d}{dt}(A_{0}(J_{\varepsilon}(\tilde{\textbf{u}}^{\varepsilon}+\bar{\textbf{u}}))D^{\alpha}J_{\varepsilon}\tilde{\textbf{u}}^{\varepsilon}, D^{\alpha}J_{\varepsilon}\tilde{\textbf{u}}^{\varepsilon})_{0} + \frac{1}{\varepsilon}((\mathbb{I}-\mathbb{P})D^{\alpha} v^{\varepsilon}, D^{\alpha} v^{\varepsilon})_{0}
\end{equation*}

\begin{equation}
\le c(|\tilde{\textbf{u}}^{\varepsilon}|_{\infty}, |\nabla \tilde{\textbf{u}}^{\varepsilon}|_{\infty}, \bar{\rho})||D^{m}\tilde{\textbf{u}}^{\varepsilon}||_{0}^{2}.
\end{equation}

Using the Hodge decomposition theorem, we can obtain the main feature of our approximation, namely
\begin{equation}
\label{Hodge1}
((\mathbb{I}-\mathbb{P})D^{\alpha} v^{\varepsilon}, D^{\alpha} v^{\varepsilon})_{0} = ((\mathbb{I}-\mathbb{P})D^{\alpha} v^{\varepsilon}, D^{\alpha} ((\mathbb{I}-\mathbb{P})v^{\varepsilon}+\mathbb{P}v^{\varepsilon}))_{0}= ||(\mathbb{I}-\mathbb{P})v^{\varepsilon}||_{0}^{2}.
\end{equation}

Equalities in (\ref{Hodge1}) imply that the singular term $$\displaystyle \frac{1}{\varepsilon}((\mathbb{I}-\mathbb{P})D^{\alpha} v^{\varepsilon}, D^{\alpha} v^{\varepsilon})_{0}=\frac{||(\mathbb{I}-\mathbb{P})v^{\varepsilon}||_{0}^{2}}{\varepsilon}$$ is positive. Summing up to $|\alpha| \le m$, we obtain

\begin{equation}
\frac{d}{dt}\sum_{|\alpha| \le m} (A_{0}(J_{\varepsilon}(\tilde{\textbf{u}}^{\varepsilon}+\bar{\textbf{u}}))D^{\alpha}J_{\varepsilon}\tilde{\textbf{u}}^{\varepsilon}, D^{\alpha}J_{\varepsilon}\tilde{\textbf{u}}^{\varepsilon})_{0} \le c(|\tilde{\textbf{u}}^{\varepsilon}|_{\infty}, |\nabla \tilde{\textbf{u}}^{\varepsilon}|_{\infty}, \bar{\rho})||\tilde{\textbf{u}}^{\varepsilon}||_{m}^{2}.
\end{equation}

Since $\displaystyle A_{0}$ is positive definite and using the properties of mollifiers, last estimate gives

\begin{equation}
\frac{d}{dt}||\tilde{\textbf{u}}^{\varepsilon}||_{m}^{2} \le c(\tilde{\textbf{u}}^{\varepsilon}|_{\infty}, |\nabla \tilde{\textbf{u}}^{\varepsilon}|_{\infty}, \bar{\rho}) ||\tilde{\textbf{u}}^{\varepsilon}||_{m}^{2}.
\end{equation}

Using the Sobolev embedding theorem and defining  $\displaystyle y(t):=||\tilde{\textbf{u}}^{\varepsilon}||^{2}_{m},$ it yields

\begin{equation}
\label{Pre_Estimate}
y'(t) \le c(|\tilde{\textbf{u}}^{\varepsilon}|_{\infty}, |\nabla \tilde{\textbf{u}}^{\varepsilon}|_{\infty} , \bar{\rho}) y(t).
\end{equation}

Now, we want to show that there exists $T$ independent of $\varepsilon$, such that $ 0 < T \le T_{\varepsilon}$ for every $\varepsilon > 0$. It is known that there exists a constant $M$ such that  $\displaystyle y(0)=||\tilde{\textbf{u}}^{\varepsilon}_{0}||_{m} \le M$. Fix a positive constant $\tilde{M} > c_{S}M$ where, again, $c_{S}$ is the constant of the Sobolev embedding, and let $T_{0}^{\varepsilon}$ be such that $0 \le T_{0}^{\varepsilon} \le T_{\varepsilon}$ and

\begin{equation}
\label{infinity_bound}
sup_{~ 0 \le \tau \le T_{0}^{\varepsilon} ~} max \{|\tilde{\textbf{u}}^{\varepsilon}(\tau)|_{\infty}, |\nabla \tilde{\textbf{u}}^{\varepsilon}(\tau)|_{\infty} \} \le c_{S}~sup_{~ 0 \le \tau \le T_{0}^{\varepsilon} ~} ||\tilde{\textbf{u}}^{\varepsilon}(\tau)||_{m} \le \tilde{M}.
\end{equation}

Thus, estimate (\ref{Pre_Estimate}) yields 

\begin{equation}
y'(t) \le c(\tilde{M}, \bar{\rho}) y(t) ~~~~ \text{for} ~ t \in [0, T_{0}^{\varepsilon}],
\end{equation}
i.e. 
\begin{equation}
\label{prebound}
||\tilde{\textbf{u}}^{\varepsilon}(t)||_{m} \le ||\tilde{\textbf{u}}^{\varepsilon}_{0}||_{m} e^{c(\tilde{M}, \bar{\rho})t} \le c_{S}M e^{c(\tilde{M}, \bar{\rho})t} ~~~~ \text{for} ~ t \in [0, T_{0}^{\varepsilon}].
\end{equation}

Now, we find $T,$ with $0 < T \le T_{0}^{\varepsilon},$ such that

\begin{equation}
c_{S}M e^{c(\tilde{M}, \bar{\rho})T} \le \tilde{M}.
\end{equation}

This way, we get

\begin{equation}
\label{time_bound}
T \le \frac{log(\frac{\tilde{M}}{Mc_{S}})}{c(\tilde{M}, \bar{\rho})}.
\end{equation}

Now, the constants $M, \tilde{M}, c_{S}, \bar{\rho}$ are independent of the parameter $\varepsilon$, therefore, estimate (\ref{time_bound}) implies that the time $T$ of existence of solutions to problem (\ref{Goal_System_Translated_Approx1}) is independent of $\varepsilon$ and

\begin{equation}
\label{1bound}
||\tilde{\textbf{u}}^{\varepsilon}(t)||_{m} \le c(M, \tilde{M}) ~~~~ \text{for} ~ t \in [0, T],
\end{equation}

provided that $\displaystyle T \le \frac{log(\frac{\tilde{M}}{Mc_{S}})}{c(\tilde{M}, \bar{\rho})}$.
\end{proof}

To obtain a uniform bound for the time derivatives $(\partial_{t}\tilde{\textbf{u}}^{\varepsilon})(t)$, in the low norm $L^{2}$, we proceed in a similar way. We take the time derivative of equation (\ref{Goal_System_Translated_Approx1}) and let $$\textbf{w}^{\varepsilon}=\partial_{t}\tilde{\textbf{u}}^{\varepsilon}=(\partial_{t}\rho^{\varepsilon}, \partial_{t}v^{\varepsilon}).$$ Then, we have

\begin{equation*}
\partial_{t}\textbf{w}^{\varepsilon}+\sum_{j=1}^{d} A_{0}^{-1}J_{\varepsilon}A_{0}A_{j}(J_{\varepsilon}(\tilde{\textbf{u}}^{\varepsilon}+\bar{\textbf{u}}))\partial_{x_{j}}J_{\varepsilon}\textbf{w}^{\varepsilon}+\sum_{j=1}^{d} (A_{0}^{-1}J_{\varepsilon}A_{0}A_{j}(J_{\varepsilon}(\tilde{\textbf{u}}^{\varepsilon}+\bar{\textbf{u}})))'J_{\varepsilon}\textbf{w}^{\varepsilon} \partial_{x_{j}}J_{\varepsilon}\tilde{\textbf{u}}^{\varepsilon}
\end{equation*}

\begin{equation}
=- \left( \begin{array}{c}
0 \\
 \frac{(\mathbb{I}-\mathbb{P})\partial_{t} v^{\varepsilon}}{\varepsilon} \\
\end{array} \right).
\end{equation}

Taking the $(A_{0} \cdot, \cdot)_{0}$ inner product with ${\textbf{w}}^{\varepsilon}$, we have

\begin{equation*}
\frac{1}{2}\frac{d}{dt}(A_{0}(J_{\varepsilon}(\tilde{\textbf{u}}^{\varepsilon}+\bar{\textbf{u}}))\textbf{w}^{\varepsilon}, \textbf{w}^{\varepsilon})_{0} + \frac{||(\mathbb{I}-\mathbb{P})\partial_{t}v^{\varepsilon}||_{0}^{2}}{\varepsilon}
\end{equation*}

\begin{equation*}
=\frac{1}{2}((A_{0}(J_{\varepsilon}(\tilde{\textbf{u}}^{\varepsilon}+\bar{\textbf{u}})))' J_{\varepsilon}\textbf{w}^{\varepsilon} \cdot \textbf{w}^{\varepsilon}, \textbf{w}^{\varepsilon})_{0}+
\frac{1}{2}\sum_{j=1}^{d} ((A_{0}A_{j}(J_{\varepsilon}(\tilde{\textbf{u}}^{\varepsilon}+\bar{\textbf{u}})))'  \partial_{x_{j}}J_{\varepsilon}\tilde{\textbf{u}}^{\varepsilon} \cdot J_{\varepsilon}\textbf{w}^{\varepsilon}, J_{\varepsilon}\textbf{w}^{\varepsilon})_{0} 
\end{equation*}

\begin{equation}
+ \sum_{j=1}^{d} (A_{0}(A_{0}^{-1}J_{\varepsilon} A_{j}(J_{\varepsilon}(\tilde{\textbf{u}}^{\varepsilon}+\bar{\textbf{u}})))' J_{\varepsilon}\textbf{w}^{\varepsilon} \partial_{x_{j}}J_{\varepsilon}\tilde{\textbf{u}}^{\varepsilon}, \textbf{w}^{\varepsilon})_{0}.
\end{equation}

Therefore, we have

\begin{equation*}
\frac{d}{dt}(A_{0}(J_{\varepsilon}(\tilde{\textbf{u}}^{\varepsilon}+\bar{\textbf{u}}))\textbf{w}^{\varepsilon}, \textbf{w}^{\varepsilon})_{0} \le c(|\tilde{\textbf{u}}^{\varepsilon}|_{\infty}, |\nabla \tilde{\textbf{u}}^{\varepsilon}|_{\infty}, \bar{\rho})||\textbf{w}^{\varepsilon}||_{0}^{2}
\end{equation*}

Since (\ref{infinity_bound}) and (\ref{1bound}), it holds

\begin{equation*}
\frac{d}{dt}(A_{0}(J_{\varepsilon}(\tilde{\textbf{u}}^{\varepsilon}+\bar{\textbf{u}}))\textbf{w}^{\varepsilon}, \textbf{w}^{\varepsilon})_{0} \le c(M, \tilde{M}, \bar{\rho})||\textbf{w}^{\varepsilon}||_{0}^{2}.
\end{equation*}

Then, since the positivity of matrix $A_{0}$ and by using the Gronwall inequality and properties of mollifiers, it follows that

\begin{equation}
||\textbf{w}^{\varepsilon}(t)||_{0}^{2} \le ||\textbf{w}^{\varepsilon}(0)||_{0}^{2} e^{c(M, \tilde{M}, \bar{\rho})t},
\end{equation}

i.e.

\begin{equation}
\label{time_derivative_bound1}
||\partial_{t}\tilde{\textbf{u}}^{\varepsilon}(t)||_{0}^{2} \le ||\partial_{t}\tilde{\textbf{u}}^{\varepsilon}(0)||_{0}^{2} e^{c(M, \tilde{M}, \bar{\rho})t}.
\end{equation}

Therefore, $(\partial_{t}\tilde{\textbf{u}}^{\varepsilon})_{\varepsilon \ge 0}$ is uniformly bounded in $L^{2}(\mathbb{R}^{d})$ for each $t \in [0,T]$, provided that $||\textbf{w}^{\varepsilon}(0)||_{0}^{2} = ||\partial_{t}\tilde{\textbf{u}}^{\varepsilon}(0)||_{0}^{2}$ is uniformly bounded in $\varepsilon$. This is guaranteed by the structural conditions on the initial data in (\ref{approximating_initial}). In fact, since (\ref{Goal_System_Translated_Approx1}), we have

\begin{equation*}
\partial_{t}^{\varepsilon}\tilde{\textbf{u}}^{\varepsilon}(0)=-\sum_{j=1}^{d}A_{0}^{-1}J_{\varepsilon}A_{0}A_{j}(J_{\varepsilon}(\tilde{\textbf{u}}_{0}^{\varepsilon}+\bar{\textbf{u}}))\partial_{x_{j}}J_{\varepsilon}\tilde{\textbf{u}}_{0}^{\varepsilon} - \left( \begin{array}{c}
0 \\
 \frac{(\mathbb{I}-\mathbb{P})v^{\varepsilon}_{0}}{\varepsilon} \\
\end{array} \right).
\end{equation*}

Recalling the structural conditions on the initial data in (\ref{approximating_initial}), we notice that

\begin{equation*}
v^{\varepsilon}_{0}(x)=v_{0}(x)+\varepsilon v_{0}^{1}(x),
\end{equation*}

where $\nabla \cdot v_{0}(x) = 0$. This means that $\mathbb{P}v_{0} = v_{0}$, and then 

\begin{equation*}
\frac{1}{\varepsilon}(\mathbb{I}-\mathbb{P})v_{0} = 0. 
\end{equation*}

Therefore

\begin{equation*}
\partial_{t}\tilde{\textbf{u}}^{\varepsilon}(0)=-\sum_{j=1}^{d}A_{0}^{-1}J_{\varepsilon}A_{0}A_{j}(J_{\varepsilon}(\tilde{\textbf{u}}_{0}^{\varepsilon}+\bar{\textbf{u}}))\partial_{x_{j}}J_{\varepsilon}\tilde{\textbf{u}}_{0}^{\varepsilon} -(\mathbb{I}-\mathbb{P})v_{0}^{1}(x),
\end{equation*}

and, by the properties of pseudodifferential operators and the Sobolev embedding theorem, we have

\begin{equation}
\label{2bound}
||\partial_{t}\tilde{\textbf{u}}^{\varepsilon}(0)||_{0} \le c(||\tilde{\textbf{u}}_{0}^{\varepsilon}||_{m}) \le c(M, \tilde{M}).
\end{equation}

which is the desired bound.

\subsection{Convergence to the solution to the Compressible-Incompressible System}
We want to prove the following theorem:

\begin{theorem}
\label{Convergence1}
Let $\tilde{\textbf{u}}_{0}=(\tilde{\rho}_{0}, \tilde{v}_{0})^{T}$ be the translated initial data in (\ref{translated_initial_data}), $\tilde{\textbf{u}}_{0} \in H^{m}(\mathbb{R}^{d})$ with $m > [d/2]+1$. There is a positive time $T$, such that there exists the unique solution $\tilde{\textbf{u}} \in C([0,T], H^{m}(\mathbb{R}^{d})) \cap C^{1}([0,T], H^{m-1}(\mathbb{R}^{d}))$ to the compressible-incompressible system (\ref{Goal_System_Translated}). The solution $\tilde{\textbf{u}}$ to equations (\ref{Goal_System_Translated}) is the limit of a subsequence of solutions to the approximating system (\ref{Goal_System_Translated_Approx1}) with initial data (\ref{Compact_translated_initial_data}). The incompressible pressure term $P$ satisfies (\ref{Goal_System_Translated}), namely

\begin{equation}
\label{incompressible_pressure}
\partial_{t}{\textbf{u}}+\sum_{j=1}^{d} A_{j}(\textbf{u}+\bar{\textbf{u}})\partial_{x_{j}}\textbf{u} + (0, \nabla P)^{T} = 0.
\end{equation}
\end{theorem}

\begin{proof}
The first part is exactly what we have done in the proof of convergence of the previous approximation. We omit that, then, taking a subsequence $(\tilde{\textbf{u}}^{\varepsilon})_{\varepsilon \ge 0}$ and its limit function $\tilde{\textbf{u}}^{\star}$, we start from some facts:

\item \begin{equation}
\label{uniform_convergence}
\tilde{\textbf{u}}^{\varepsilon} \rightarrow \tilde{\textbf{u}}^{\star} ~~ \text{as} ~~  \varepsilon \rightarrow 0 ~~\text{in} ~~C([0,T], H^{m'}_{loc}(\mathbb{R}^{d})) ~~ {m}' <  {m};
\end{equation}

 \begin{equation}
\label{L_2_weakly}
\tilde{\textbf{u}}^{\varepsilon} \rightharpoonup \tilde{\textbf{u}}^{\star} ~~ \text{as} ~~  \varepsilon \rightarrow 0 ~~\text{in} ~~L^{2}([0,T], H^{m}(\mathbb{R}^{d}));
\end{equation}

\begin{equation}
\label{L_infty}
\tilde{\textbf{u}}^{\star} \in L^{\infty}([0,T], H^{m}(\mathbb{R}^{d}));
\end{equation}

\begin{equation}
\label{uniform}
\tilde{\textbf{u}}^{\star} \in C([0,T], H^{m'}(\mathbb{R}^{d})) ~~ {m}' <  {m};
\end{equation}

\begin{equation}
\label{all_uniform}
\tilde{\textbf{u}}^{\star} \in C_{w}([0,T], H^{m}(\mathbb{R}^{d})).
\end{equation}

\par\bigskip

Now, recalling that

\begin{equation}
A_{0}(J_{\varepsilon}(\tilde{\textbf{u}}^{\varepsilon}+\bar{\textbf{u}}))\partial_{t}\tilde{\textbf{u}}^{\varepsilon}+\sum_{j=1}^{d}A_{0}A_{j}(J_{\varepsilon}(\tilde{\textbf{u}}^{\varepsilon}+\bar{\textbf{u}}))\partial_{x_{j}}J_{\varepsilon}\tilde{\textbf{u}}^{\varepsilon} = - \left( \begin{array}{c}
0 \\
 \frac{(\mathbb{I}-\mathbb{P})v^{\varepsilon}}{\varepsilon} \\
\end{array} \right),
\end{equation}

and looking at (\ref{1bound}) and (\ref{2bound}), we have

\begin{equation}
\label{PressureBound}
sup_{~0\le t \le T~}\frac{1}{\varepsilon} ||(\mathbb{I}-\mathbb{P})v^{\varepsilon}||_{0} \le c(M, \tilde{M}).
\end{equation}

This means that $||(\mathbb{I}-\mathbb{P})v^{\varepsilon}||_{0}\rightarrow 0 $ as $\varepsilon \rightarrow 0$ and, since $v^{\varepsilon} \rightarrow \tilde{v}^{\star}$ in $C([0,T], H^{m'}_{loc}(\mathbb{R}^{d}))$, it follows that $\mathbb{P}\tilde{v}^{\star} = \tilde{v}^{\star}$, namely

\begin{equation}
\label{DivNull}
\nabla \cdot \tilde{v}^{\star} = 0. 
\end{equation}

Next, let $\psi \in C_{c}^{\infty}((0,T))$ and $\phi=(\rho, v)^{T}$ so that  $v \in \mathcal{V}^{0}=\{ v \in L^{2}(\mathbb{R}^{d}) ~|~ \nabla \cdot v = 0\}$ rapidly decreasing. Writing a weak formulation of system (\ref{Goal_System_Translated_Approx1}), we have

\begin{equation*}
\int_{0}^{T} - \psi'(t) (\tilde{\textbf{u}}^{\varepsilon}, \phi)_{0}  ~ dt  + \sum_{j=1}^{d}\int_{0}^{T} \psi(t) (A_{0}^{-1}J_{\varepsilon}A_{0}A_{j}(J_{\varepsilon}(\tilde{\textbf{u}}^{\varepsilon}+\bar{\textbf{u}}))\partial_{x_{j}}J_{\varepsilon}\tilde{\textbf{u}}^{\varepsilon},  \phi)_{0}  ~ dt 
\end{equation*}

\begin{equation}
\label{weak_solution}
= - \int_{0}^{T} {\psi(t) }\Bigg(\frac{(\mathbb{I}-\mathbb{P})v^{\varepsilon}}{\varepsilon}, v\Bigg)_{0} ~ dt.
\end{equation}

Since $(\mathbb{I}-\mathbb{P})v^{\varepsilon}$ is a gradient for every $\varepsilon$, the right hand side of last equality vanishes.
Then, equation (\ref{weak_solution}) becomes

\begin{equation}
\label{weak_solution_projection}
\int_{0}^{T} - \psi'(t) (\tilde{\textbf{u}}^{\varepsilon}, \phi)_{0}  ~ dt  + \sum_{j=1}^{d}\int_{0}^{T} \psi(t) (A_{0}^{-1}J_{\varepsilon}A_{0}A_{j}(J_{\varepsilon}(\tilde{\textbf{u}}^{\varepsilon}+\bar{\textbf{u}}))\partial_{x_{j}}J_{\varepsilon}\tilde{\textbf{u}}^{\varepsilon},  \phi)_{0}  ~ dt = 0.
\end{equation}

By using the uniform convergence of $\tilde{\textbf{u}}^{\varepsilon} \rightarrow \tilde{\textbf{u}}^{\star}$ in $C([0,T], H^{m'}_{loc}(\mathbb{R}^{d}))$ and recalling that $\tilde{\textbf{u}}^{\star} \in L^{\infty}([0,T], H^{m}(\mathbb{R}^{d}))$, we have that  

\begin{equation}
(\tilde{\textbf{u}}^{\varepsilon}, \phi) \rightarrow (\tilde{\textbf{u}}^{\star}, \phi)_{0}.
\end{equation}

Now, for $j=1, \cdots, d$, it holds 
$$|(A_{0}^{-1}J_{\varepsilon}A_{0}A_{j}(J_{\varepsilon}(\tilde{\textbf{u}}^{\varepsilon}+\bar{\textbf{u}}))\partial_{x_{j}}J_{\varepsilon}\tilde{\textbf{u}}^{\varepsilon}-A_{j}(\tilde{\textbf{u}}^{\star}+\bar{\textbf{u}})\partial_{x_{j}}\tilde{\textbf{u}}^{\star}, \phi)_{0}|$$

$$\le |(A_{0}^{-1}J_{\varepsilon}A_{0}A_{j}(J_{\varepsilon}(\tilde{\textbf{u}}^{\varepsilon}+\bar{\textbf{u}}))\partial_{x_{j}}(J_{\varepsilon}\tilde{\textbf{u}}^{\varepsilon}-\tilde{\textbf{u}}^{\star}), \phi)_{0}|$$

$$+|((A_{0}^{-1}J_{\varepsilon}A_{0}A_{j}(J_{\varepsilon}(\tilde{\textbf{u}}^{\varepsilon}+\bar{\textbf{u}}))-A_{j}(\tilde{\textbf{u}}^{\star}+\bar{\textbf{u}}))\partial_{x_{j}}\tilde{\textbf{u}}^{\star}, \phi)_{0}|$$

\begin{equation}
\le M_{1} |(\partial_{x_{j}}(J_{\varepsilon}\tilde{\textbf{u}}^{\varepsilon}-\tilde{\textbf{u}}^{\star}), \phi)_{0}| + |((J_{\varepsilon} A_{j}(J_{\varepsilon}(\tilde{\textbf{u}}^{\varepsilon}+\bar{\textbf{u}}))-A_{j}(\tilde{\textbf{u}}^{\varepsilon}+\bar{\textbf{u}}))\partial_{x_{j}}\tilde{\textbf{u}}^{\star}, \phi)_{0}|.
\end{equation}

Therefore, again, since
$\tilde{\textbf{u}}^{\varepsilon} \rightarrow \tilde{\textbf{u}}^{\star}$ in $C([0,T], H^{m'}_{loc}(\mathbb{R}^{d}))$ and using that $\tilde{\textbf{u}}^{\star} \in L^{\infty}([0,T], H^{m}(\mathbb{R}^{d}))$, it follows the convergence

\begin{equation}
\label{L2convergence}
\int_{0}^{T} \psi(t) (A_{0}^{-1}J_{\varepsilon}A_{0}A_{j}(J_{\varepsilon}(\tilde{\textbf{u}}^{\varepsilon}+\bar{\textbf{u}}))\partial_{x_{j}}J_{\varepsilon}\tilde{\textbf{u}}^{\varepsilon},  \phi)_{0} \rightarrow\int_{0}^{T} \psi(t) (A_{j}(\tilde{\textbf{u}}^{\star}+\bar{\textbf{u}})\partial_{x_{j}}\tilde{\textbf{u}}^{\star},  \phi)_{0}
\end{equation}

uniformly on $[0,T]$. This way, passing to the limit in (\ref{weak_solution_projection}), we obtain 

\begin{equation}
\label{weak_solution_limit}
\int_{0}^{T} - \psi'(t) (\tilde{\textbf{u}}^{\star}, \phi)_{0}  ~ dt  + \sum_{j=1}^{d}\int_{0}^{T} \psi(t) (A_{j}(\tilde{\textbf{u}}^{\star}+\bar{\textbf{u}})\partial_{x_{j}}\tilde{\textbf{u}}^{\star},  \phi)_{0}  ~ dt = 0.
\end{equation}

Again, convergence of (\ref{weak_solution_projection}) to (\ref{weak_solution_limit}) and (\ref{2bound}) yields $\partial_{t}\tilde{\textbf{u}}^{\varepsilon} \rightharpoonup^{*}\partial_{t}\tilde{\textbf{u}}^{\star}$ in $L^{\infty}([0,T], L^{2}(\mathbb{R}^{d}))$ and equation (\ref{distribution_solution_limit}), i.e.

\begin{equation}
\partial_{t}\tilde{\textbf{u}}^{\star}+\sum_{j=1}^{d}\textbf{P}(A_{j}(\tilde{\textbf{u}}^{\star}+\bar{\textbf{u}})\partial_{x_{j}}\tilde{\textbf{u}}^{\star})=0.
\end{equation}

This way, we recover the adding regularity $\tilde{\textbf{u}}^{\star} \in Lip([0,T], H^{m-1}(\mathbb{R}^{d}))$ and the existence of $\nabla P^{\star} \in L^{\infty}([0,T], H^{m-1}(\mathbb{R}^{d}))$ such that

\begin{equation}
\label{distribution_solution_limit21}
\partial_{t}\tilde{\textbf{u}}^{\star} + \sum_{j=1}^{d}A_{j}(\tilde{\textbf{u}}^{\star}+\bar{\textbf{u}})\partial_{x_{j}}\tilde{\textbf{u}}^{\star}=(0, -\nabla P^{\star})^{T}.
\end{equation}

Thus, $\tilde{\textbf{u}}^{\star} \in L^{\infty}([0,T], H^{m}(\mathbb{R}^{d})) \cap Lip([0,T], H^{m-1}(\mathbb{R}^{d})) \cap C_{w}([0,T], H^{m}(\mathbb{R}^{d}))$ is a weak solution to the compressible-incompressible system (\ref{Goal_System}). Passing to a subsequence and recalling

\begin{equation}
\partial_{t}\tilde{\textbf{u}}^{\star} \rightharpoonup^{*} \partial_{t}\tilde{\textbf{u}}^{\star} ~~ in ~~ L^{\infty}([0,T],L^{2}(\mathbb{R}^{d})),
\end{equation}

and (\ref{weak_solution_projection})-(\ref{weak_solution_limit}), we get

\begin{equation}
\label{convergence1}
\partial_{t}\tilde{\textbf{u}}^{\varepsilon} + \sum_{j=1}^{d} A_{0}^{-1}J_{\varepsilon}A_{0}A_{j}(\tilde{\textbf{u}}^{\varepsilon}+\bar{\textbf{u}})\partial_{x_{j}}J_{\varepsilon}\tilde{\textbf{u}}^{\varepsilon} \rightharpoonup^{*} \partial_{t}\tilde{\textbf{u}}^{\star} + \sum_{j=1}^{d} A_{j}(\tilde{\textbf{u}}^{\star}+\bar{\textbf{u}})\partial_{x_{j}}\tilde{\textbf{u}}^{\star}.
\end{equation}

in  $\displaystyle L^{\infty}([0,T],L^{2}(\mathbb{R}^{d}))$, where $ \partial_{t}\tilde{\textbf{u}}^{\star}  \in C_{w}([0,T], H^{m-1}(\mathbb{R}^{d}))$ since (\ref{distribution_solution_limit}) holds and $ \tilde{\textbf{u}}^{\star}  \in C_{w}([0,T], H^{m}(\mathbb{R}^{d}))$. Recalling
\begin{equation}
-\frac{1}{\varepsilon}(\mathbb{I}-\mathbb{P})v^{\varepsilon} = - {\nabla P^{\varepsilon}},
\end{equation}

and by (\ref{PressureBound}), we have

\begin{equation}
sup_{~0\le t \le T~} ||\nabla P^{\varepsilon}||_{0} \le c(M_{1}, M_{2}).
\end{equation}

Therefore, by using the Banach-Alaoglu theorem in $L^{\infty}([0,T], L^{2}(\mathbb{R}^{d}))$ and recalling (\ref{distribution_solution_limit2}), we get

\begin{equation}
\label{PressureConvergence}
{\nabla P^{\varepsilon}} \rightharpoonup^{*} \nabla P^{\star} ~~ in ~~ L^{\infty}([0,T], L^{2}(\mathbb{R}^{d})),
\end{equation}

with $P^{*} \in C_{w}([0,T], H^{m-1}(\mathbb{R}^{d}))$ since (\ref{distribution_solution_limit}) and the regularity of $\partial_{t}\tilde{\textbf{u}}^{*}$ and $\tilde{\textbf{u}}^{\star}$.

Applying exaclty - apart from a slight modification related to the different approximating initial data - the same proof of last part of Section 2, we get $ \tilde{\textbf{u}}^{\star} \in C([0,T], H^{m}(\mathbb{R}^{d})) \cap C^{1}([0,T], H^{m-1}(\mathbb{R}^{d}))$. 
\end{proof}

\begin{remark}
\label{remark3}
If we approximate the density-dependent incompressible Euler equations (\ref{density_dependent}) by this method, for instance, in the two-dimensional case, we have the following $\varepsilon$-system:
\begin{equation}
\partial_{t}\textbf{u}^{\varepsilon}+
\left( \begin{array}{ccc}
v^{\varepsilon}_{1} & 0 & 0 \\
0 & v^{\varepsilon}_{1} & 0 \\
0 & 0 & v^{\varepsilon}_{1} 
\end{array} \right)\partial_{x}\textbf{u}^{\varepsilon} + \left( \begin{array}{ccc}
v^{\varepsilon}_{2} & 0 & 0 \\
0 & v^{\varepsilon}_{2} & 0 \\
0 & 0 & v^{\varepsilon}_{2} 
\end{array} \right)\partial_{y}\textbf{u}^{\varepsilon}+ \Bigg(0, \frac{(\mathbb{I}-\mathbb{P})v^{\varepsilon}}{\varepsilon \rho^{\varepsilon}}\Bigg)^{T}=0,
\end{equation}
where $\textbf{u}^{\varepsilon}:=(\rho^{\varepsilon}, v^{\varepsilon})$. Taking the $\alpha$-derivative in order to get energy estimates, we notice that we have no more positivity of the singular term in $\varepsilon$, therefore this method does not work in a simple way on (\ref{density_dependent}).
\end{remark}

\section{Artificial Compressibility Method}
Following \cite{Temam}, we consider an other kind of approximation of system (\ref{Goal_System}), based on a family of $\varepsilon$-dependent perturbed system, which, in order to approximate the divergence constraint $\nabla \cdot v = 0$, contains the following artificial equation for the pressure term $P^{\varepsilon}$:

\begin{equation}
\label{Articial_state_equation1}
\varepsilon^{2} \partial_{t}P^{\varepsilon} + \nabla \cdot v^{\varepsilon} = 0.
\end{equation}

We consider the artificial state equation 

\begin{equation}
\label{P_expansion}
P^{\varepsilon}=P_{0}+\varepsilon \tilde{P}^{\varepsilon},
\end{equation}

where $P_{0}$ is constant. Without loss of generality, we take $P_{0}=1$. Setting $\textbf{u}^{\varepsilon}:=(\rho^{\varepsilon}, \tilde{P}^{\varepsilon}, v^{\varepsilon})$, the approximating system becomes:

\begin{equation}
\label{Goal_System_Incompressible}
\begin{cases}
& \partial_{t}\rho^{\varepsilon} + \nabla \cdot (\rho^{\varepsilon} v^{\varepsilon}) = 0, \\
& \partial_{t}\tilde{P}^{\varepsilon} + \frac{\nabla \cdot v^{\varepsilon}}{\varepsilon} = 0,\\
& \partial_{t}{v^{\varepsilon}}+{v^{\varepsilon}} \cdot \nabla {v^{\varepsilon}} + f(\rho^{\varepsilon}, v^{\varepsilon}) \nabla \rho^{\varepsilon} + \frac{\nabla {\tilde{P}^{\varepsilon}}}{\varepsilon}=0, 
\end{cases}
\end{equation}

with the following initial data as in (\ref{approximating_initial}):

\begin{equation}
\rho_{0}^{\varepsilon}(x)=\rho_{0}(0,x), ~~~~  v_{0}^{\varepsilon}(0, x)=v_{0}(x)+\varepsilon v_{0}^{1}(x),
\end{equation}

where $\rho_{0}, v_{0}$ are the initial data (\ref{initial}) of the original problem (\ref{Goal_System}).

\begin{remark}
To simplify the notation, we are skipping the translation of the density variable $\rho$, which is required also for system (\ref{density_dependent}) in its compact form, because of the fact that it is the same argument previously discussed, see Remark \ref{Translation} and Remark \ref{Transalted_system}.
\end{remark}

 Again, we can write system (\ref{Goal_System_Incompressible}) in the compact form:

\begin{equation}
\label{Compact_Incompressible}
\partial_{t}\textbf{u}^{\varepsilon} + \sum_{j=1}^{d} A_{j}(\textbf{u}^{\varepsilon})\partial_{x_{j}}\textbf{u}^{\varepsilon} = 0,
\end{equation}

with initial data

\begin{equation}
\label{Compact_incompressible_initial_data}
\textbf{u}^{\varepsilon}_{0}=(\rho^{\varepsilon}_{0}, \tilde{P}^{\varepsilon}_{0}, v^{\varepsilon}_{0})^{T},
\end{equation}

where the function $\tilde{P}_{0}$ is not given with problem (\ref{Goal_System}), but it is arbitrarily chosen, provided that $\tilde{P}_{0} \in H^{m}(\mathbb{R}^{d})$ and, as in (\ref{approximating_initial}), $$\rho^{\varepsilon}_{0}=\rho_{0},~~ ~ ~~ v^{\varepsilon}_{0}=v_{0}+\varepsilon v_{0}^{1},$$ with $\rho_{0}, v_{0}$ the original initial data in (\ref{initial}). The matrices $A_{j}(\textbf{u}^{\varepsilon})$ have the following structural form:

\begin{equation}
\label{Aj_incompressible}
A_{j}(\textbf{u}^{\varepsilon})=\tilde{A}_{j}(\textbf{u}^{\varepsilon}) + \frac{{A}^{0}_{j}}{\varepsilon}
\end{equation}

for $j=1, \cdots, d$. In the $2$-dimensional case, we have

\begin{equation}
\label{A1}
A_{1}(\textbf{u}^{\varepsilon}) = \tilde{A}_{1}(\textbf{u}^{\varepsilon}) + \frac{{A}^{0}_{1}}{\varepsilon} =  \left( \begin{array}{cccc}
v^{\varepsilon}_{1} & 0 & \rho^{\varepsilon} & 0 \\
0 & 0 & 0 & 0  \\
f(\textbf{u}^{\varepsilon}) & 0 & v^{\varepsilon}_{1} & 0 \\
0 & 0 & 0 & v^{\varepsilon}_{1} 
\end{array} \right) +  \left( \begin{array}{cccc}
0 & 0 & 0 & 0 \\
0 & 0 & \frac{1}{\varepsilon} & 0  \\
0 & \frac{1}{\varepsilon} & 0 & 0 \\
0 & 0 & 0 & 0
\end{array} \right),
\end{equation}

\begin{equation}
\label{A2}
A_{2}(\textbf{u}^{\varepsilon}) = \tilde{A}_{2}(\textbf{u}^{\varepsilon}) + \frac{{A}^{0}_{2}}{\varepsilon} =  \left( \begin{array}{cccc}
v^{\varepsilon}_{2} & 0 & 0 & \rho^{\varepsilon} \\
0 & 0 & 0 & 0  \\
0 & 0 & v^{\varepsilon}_{2} & 0\\
f(\textbf{u}^{\varepsilon}) & 0 & 0 & v^{\varepsilon}_{2}
\end{array} \right) +  \left(\begin{array}{cccc}
0 & 0 & 0 & 0 \\
0 & 0 & 0 & \frac{1}{\varepsilon}  \\
0 & 0 & 0 & 0 \\
0 & \frac{1}{\varepsilon} & 0 & 0
\end{array} \right)
\end{equation}

and, in the general $d$-case, for $j=1, \cdots, d$

$$A_{j}(\textbf{u}^{\varepsilon})=\tilde{A}_{j}(\textbf{u}^{\varepsilon})+\frac{A^{0}_{j}}{\varepsilon}$$

\begin{equation}
\label{Aj}
= \left( \begin{array}{cccccc}
v^{\varepsilon}_{j} & 0 & \delta_{1j}\rho^{\varepsilon} & \delta_{2j}\rho^{\varepsilon} & \cdots & \delta_{dj}\rho^{\varepsilon}\\
0 & 0 & 0 & \cdots & \cdots & 0 \\
\delta_{j1} f(\textbf{u}^{\varepsilon}) & v^{\varepsilon}_{j} & 0 & \cdots & \cdots & 0 \\
\delta_{j2} f(\textbf{u}^{\varepsilon}) & 0 & v^{\varepsilon}_{j} & \cdots & \cdots & 0 \\
\cdots &  \cdots & \cdots &  v^{\varepsilon}_{j} & \cdots & \cdots \\
\cdots &  \cdots & \cdots &  \cdots & v^{\varepsilon}_{j} & \cdots \\
\delta_{jd}f(\textbf{u}^{\varepsilon}) & 0 & 0 & \cdots & \cdots & v^{\varepsilon}_{j} 
\end{array} \right) 
+ \left( \begin{array}{cccccc}
0 & 0 & 0 & \cdots & \cdots & 0\\
0 & 0 & \frac{\delta_{1j}}{\varepsilon} &  \frac{\delta_{2j}}{\varepsilon}& \cdots &  \frac{\delta_{dj}}{\varepsilon} \\
0 & \frac{\delta_{1j}}{\varepsilon} & 0 & 0 & \cdots & 0 \\
0 & \frac{\delta_{2j}}{\varepsilon} & 0 & \cdots & \cdots & 0 \\
\cdots &  \cdots & \cdots & \cdots & \cdots & \cdots \\
\cdots &  \cdots & \cdots &  \cdots & \cdots & \cdots \\
\cdots &  \frac{\delta_{dj}}{\varepsilon}  & 0 & \cdots & \cdots & \cdots
\end{array} \right).
\end{equation}

System (\ref{Compact_Incompressible}) is Friedrichs-symmetrizable by the $(d+2) \times (d+2)$ - symmetrizer

\begin{equation}
\label{Symmetrizer}
A_{0}(\textbf{u}^{\varepsilon})=diag\Bigg(\frac{f(\textbf{u}^{\varepsilon})}{{\rho}^{\varepsilon}}, 1, 1, \cdots, 1\Bigg).
\end{equation}

Now, looking at the matrices $A_{j}$ for  $j=1, \cdots, d$, we notice that they satisfy the structural conditions required by \emph{Majda} and \emph{Klainerman} in \cite{Majda} and \cite{Klainerman} to prove the convergence of the compressible Euler equations to the incompressible ones. Moreover, the initial data (\ref{Compact_incompressible_initial_data}) associated to system (\ref{Compact_Incompressible}) are consistent with respect to the hypothesis of 'slightly compressible initial data' in \cite{Majda}. Therefore, according to the incompressible limit in \cite{Majda}, the proof of convergence of a solution to system (\ref{Compact_Incompressible}) to a solution to the goal system (\ref{Goal_System}) is straightforward here, provided, as we have already point out, the apposite translation on the variable $\rho$ that we have made in Section 2.

\begin{remark}
\label{remark2}
Applying the artificial compressibility method to system (\ref{density_dependent}), we obtain an approximation system whose matrices and the related Friedrichs symmetrizer do not satisfy the assumptions stated in \cite{Majda}. For instance, in the two-dimensional case, setting $\textbf{u}^{\varepsilon}:=(\rho^{\varepsilon}, \tilde{P}^{\varepsilon}, v^{\varepsilon})$, we have the system
\begin{equation}
\partial_{t}\textbf{u}^{\varepsilon}+
\left( \begin{array}{cccc}
v^{\varepsilon}_{1} & 0 & 0 & 0 \\
0 & 0 & \frac{1}{\varepsilon} & 0 \\
0 & \frac{1}{\varepsilon \rho^{\varepsilon}} & v^{\varepsilon}_{1} & 0 \\
0 & 0 & 0 & v^{\varepsilon}_{1}
\end{array} \right)\partial_{x}\textbf{u}^{\varepsilon} + \left( \begin{array}{cccc}
v^{\varepsilon}_{2} & 0 & 0 & 0 \\
0 & 0 & 0 & \frac{1}{\varepsilon}\\
0 & 0 & v^{\varepsilon}_{2} & 0 \\
0 & \frac{1}{\varepsilon \rho^{\varepsilon}} & 0 & v^{\varepsilon}_{2}
\end{array} \right)\partial_{x}\textbf{u}^{\varepsilon}=0,
\end{equation}

where the $\varepsilon$-singular parts of $A_{1}, A_{2}$ are not constant. Therefore, this approximation does not work on system (\ref{density_dependent}).

\end{remark}


\begin{thebibliography}{11}

\bibitem{astanin}
\textsc{Astanin S. \& Preziosi L.} (2008)
\newblock Multiphase models of tumor growth., 
Selected topics in cancer modeling,
\newblock {\em Model. Simul. Sci. Eng. Technol., }
Birkh\"auser  Boston, Boston, MA, 
223--253.

\bibitem{Da Veiga}
\textsc{Beir\~{a}o da Veiga H. \& Valli A.} (1980)
\newblock Existence of $C^{\infty}$ solutions of the Euler Equations for non-homogeneous fluids,
\newblock  {\em Comm. Part. Diff. Eq.}
5,
95-107.

\bibitem{Benzoni}
\textsc{Benzoni-Gavage S. \& Serre D.} (2007)
\newblock Multidimensional Hyperbolic Partial Differential Equations,
\newblock{\em Oxford University Press.}

\bibitem{Bertozzi}
\textsc{Bertozzi A. \& Majda A.} (2002)
\newblock Vorticity and Incompressible Flow,
\newblock{\em Cambridge University Press.}

\bibitem{Bianchini}
\textsc{Bianchini R. \& Natalini R.} (2016)
\newblock Global existence and asymptotic stability of smooth solutions to a fluid dynamics model of biofilms in one space dimension,
\newblock {\em J. Math. Anal. Appl.,}
{434},
1909-1923.

\bibitem{cdnr}
\textsc{Clarelli F., Di Russo C., Natalini R. \& Ribot M.} (2013)
\newblock A fluid dynamics model of the growth of phototrophic biofilms,
\newblock  {\em J. Math. Biol.,}
{66}(7),
1387--1408.

\bibitem{Danchin}
\textsc{Danchin R.} (2010)
\newblock On the well-posedness of the incompressible density-dependent Euler equations in the $L^{p}$ framework,
\newblock  {\em J. Diff. Eq.}
{24}(8),
2130-2170.

\bibitem{Farina}
\textsc{Farina A. \& Preziosi L.} (2002)
\newblock On Darcy's law for growing porous media,
\newblock  {\em In. J. Non-Lin. Mech.}
37(3),
485-491.

\bibitem{Grenier}
\textsc{Grenier E.} (1997)
\newblock Pseudo-Differential Energy Estimates of Singular Perturbations,
\newblock  {\em Comm. Pure Appl. Math.}
{50}(9),
821-865.
 
\bibitem{Klainerman}
\textsc{Klainerman S. \& Majda A.} (1981)
\newblock Singular Limits of Quasilinear Hyperbolic Systems with Large Parameters and the Incompressible Limit of Compressible Fluids,
\newblock {\em Comm. Pure Appl. Math.}
{XXXIV},
481-524.

\bibitem{Lad}
\textsc{Ladyzhenskaya A. O. \& Solonnikov V. A.} (1978)
\newblock Unique solvability of an initial and boundary value problem for viscous incompressible non-homogeneous fluids,
\newblock {\em Clarendon Press, Oxford}.
9,
697–749.

\bibitem{Lions}
\textsc{Lions P. L.} (1996)
\newblock Mathematical Topics in Fluid Mechanics,
\newblock {\em J. Soviet. Math.}

\bibitem{Majda}
\textsc{Majda A.} (1984)
\newblock Compressible Fluid Flow and Systems of Conservation Laws in Several
Space Variables, 
\newblock {\em Springer-Verlag, New York}.

\bibitem{Marsden}
\textsc{Marsden J. E.} (1976)
\newblock Well-posedness of the equations of a non-homogeneous perfect fluid,
\newblock {\em Comm. Part. Diff. Eq.}
1,
215-230.

\bibitem{Metivier}
\textsc{M\'etivier G.} (2008)
\newblock Paradifferential Calculus and Application to the Cauchy Problem for Nonlinear Systems,
\newblock {\em CRM Series, Edizioni della Scuola Normale Superiore}.

\bibitem{Rajagopal} 
\textsc{Rajagopal K.R. \& Tao L.} (1995)
\newblock Mechanics of Mixtures,
\newblock Series on Advances in Mathematics for Applied Sciences, 35. World Scientific Publishing Co., River Edge, NJ.       

\bibitem{Schochet}
\textsc{Schochet S.} (1986)
\newblock The Compressible Euler Equations in a Bounded Domain: Existence of Solutions and the Incompressible Limit,
\newblock {\em Communications in Mathematical Physics}
{104},
49-75.

\bibitem{Taylor}
\textsc{Taylor M.} (1996)
\newblock Partial differential equations III,
\newblock {\em Applied Mathematical Sciences 117, Springer}.  

\bibitem{Temam}
\textsc{Temam R.} (1977)
\newblock Navier-Stokes Equations -Theory and Numerical Analysis,
\newblock {\em North-Holland Publishing Company}.  

\bibitem{Valli}
\textsc{Valli A. \& Zajazckowski W. M.} (1988)
\newblock About the motion of nonhomogeneous ideal incompressible fluids,
\newblock {\em Nonlinear Analysis, Theory, Methods} \& {\em Applications}
{Vol.12, No. 1}
{43-50}.

\end{thebibliography}
\end{document}